\documentclass[11pt]{amsart}
\usepackage{amssymb, amsmath, latexsym, amsthm, amscd, fancyhdr, graphicx, xcolor, hyperref, bm}
\usepackage{nicefrac}
\def\R{\mathbb R}
\def\N{\mathbb N}
\def\C{\mathbb C}
\def\Z{\mathbb Z}
\def\Q{\mathbb Q}

\def\H{\mathbb H}

\def\={\equiv}
\def\<{\langle}
\def\>{\rangle}
\def\eps{\varepsilon}

\def\ov{\overline}

\def\inv{^{-1}}

\def\supp{\operatorname{supp}}
\def\stab{\operatorname{Stab}}

\def\bms{m^{\operatorname{BMS}}}
\def\tbms{\tilde{m}^{\operatorname{BMS}}}

\def\1{\mathbf{1}}
\def\SO{\operatorname{SO}}

\def\norm#1{\left\Vert #1\right\Vert }
\def\opt1{\operatorname{T}^1}
\def\SL{\operatorname{SL}}

\def\t{\textbf{t}}
\def\e{\textbf{e}}

\def\calh{\mathcal{H}}

\def\ps{\mu^{\operatorname{PS}}}

\def\br{m^{\operatorname{BR}}}
\def\tbr{\tilde{m}^{\operatorname{BR}}}
\def\leb{\mu^{\operatorname{Leb}}}

\def\inj{\operatorname{inj}}
\def\Stab{\operatorname{Stab}}

\newtheorem{theorem}{Theorem}[section]

\newtheorem{proposition}[theorem]{Proposition}
\newtheorem{corollary}[theorem]{Corollary}
\newtheorem{lemma}[theorem]{Lemma}
\newtheorem{definition}[theorem]{Definition}

\newtheorem{remark}[theorem]{Remark}

\title[Distribution of orbits]{Distribution of orbits of geometrically finite groups acting on null vectors}
\author[N. Tamam \and J. M. Warren]{Nattalie Tamam \and Jacqueline M. Warren}
\address[N. Tamam \and J. M. Warren]{Department of Mathematics, University of California, San Diego}

\begin{document}
\maketitle

\begin{abstract}
    We study the distribution of non-discrete orbits of geometrically finite groups in $\SO(n,1)$ acting on $\R^{n+1}$, and more generally on the quotient of $\SO(n,1)$ by a horospherical subgroup. Using equidistribution of horospherical flows, we obtain both asymptotics for the distribution of orbits for the action of general geometrically finite groups, and we obtain quantitative statements with additional assumptions.
\end{abstract}

\tableofcontents

\section{Introduction}

We often seek to understand a group through the distribution of its orbits on a given space. In this paper, we will consider the action of certain geometrically finite groups on $\R^{n+1}$ and other spaces. 

When $\Gamma$ is a lattice in $\SL_2(\R)$ acting on $\R^2$, this question was considered by Ledrappier \cite{Ledrappier}, who proved that \[\lim\limits_{T\to\infty} \frac{1}{T}\sum\limits_{\gamma \in \Gamma, \|\gamma\|\le T} f(X\gamma) = c(\Gamma)\int_{\R^2} \frac{f(Y)}{|X||Y|}dY\] for compactly supported functions $f$ and $X \in \R^2$, where $c(\Gamma)$ is some constant depending on the covolume of the lattice $\Gamma$, and $\|\gamma\|$ denotes the $\ell_2$ norm on $\Gamma$. Nogueira \cite{Nogueira} independently obtained this result for $\Gamma=\SL_2(\Z)$ using different methods. More recently, Macourant and Weiss obtained a quantitative version of this theorem for cocompact lattices in $\SL_2(\R)$, and also for $\Gamma=\SL_2(\Z)$ in \cite{MauWeiss}. The case of lattices in $\SL_n(\R)$ acting on different spaces $V$ has also been considered, see for instance \cite{Gorodnik, GorodnikMaucourant}.  

In \cite{Pollicott}, Pollicott proved a similar quantitative theorem for the action of a lattice in $\SL_2(\C)$ on $\C^2$. In the $p$-adic case, Ledrappier and Pollicott \cite{LedrappierPollicott} considered lattices in $\SL_2(\Q_p)$ acting on $\Q_p^2$.

Similar questions have been studied extensively for lattices in a wide variety of groups $G$. For instance, Gorodnik and Weiss consider in \cite{GorodnikWeiss} second countable, locally compact groups $G$ with a general axiomatic approach, with several examples. More recently, Gorodnik and Nevo comprehensively studied the action of a lattice in a connected algebraic Lie group acting on infinite volume homogeneous varieties in \cite{GorodnikNevo}, including obtaining quantitative results under appropriate assumptions. %\tcr{list additional, more general lattice case references}\cite{GorodnikOh}

The case when $\Gamma$ has infinite covolume was recently studied by Maucourant and Schapira in \cite{MauSchap}, where they obtained an asymptotic version of Ledrappier's result for convex cocompact subgroups of $\SL_2(\R)$, with a scaling factor permitted. Moreover, they prove that an ergodic theorem like Ledrappier's in the lattice case cannot be obtained in the infinite volume setting, because there is not even a ratio ergodic theorem. More specifically, \cite[Prop. 1.5]{MauSchap} shows that if $\Gamma \subseteq \SL_2(\R)$ is geometrically finite with $-I$ the unique torsion element, then there exist small bump functions $f$ and $g$ such that for $\ov\nu$-almost every $v$ (where $\ov\nu$ is defined in \S\ref{section: applications}), \[\frac{\sum_{\gamma \in \Gamma_T} f(v\gamma)}{\sum_{\gamma \in \Gamma_T} g(v\gamma)}\] does not have a limit. Thus, it is impossible to obtain an ergodic theorem in this setting with a normalization factor that does not depend on the functions. The key obstruction is the fluctuating behaviour of the Patterson-Sullivan measure. However, they show that with an additional averaging to address these fluctuations, there is a Log-Cesaro convergence, see \cite[Theorem 1.6]{MauSchap}.

Throughout this paper, let $G= \SO(n,1)^\circ$ and let $\Gamma \subseteq G$ be a Zariski dense geometrically finite subgroup. %Assume that $\Gamma$ geometrically finite. %If $\Gamma$ is geometrically finite, we will also assume that the action of the frame flow, $A$, is mixing for the Bowen-Margulis-Sullivan measure on $G/\Gamma$, specifically satisfying the statement of \cite[Assumption 1.1]{BR eqdistr}. For a discussion of when this is known to hold, see \cite[\S 8]{BR eqdistr}.
%In this paper, we study the distribution of $\Gamma$ orbits on $U\backslash G$, where $U$ is the the expanding horospherical subgroup for the frame flow $A$. 
As a consequence of a more general ratio theorem we will discuss later in this section, we will obtain the following asymptotic behaviour for $\Gamma$ orbits acting on $$V=\e_{n+1} G \setminus \{0\},$$which is similar to a result of Maucourant and Schapira for $n=2$. Note that $V$ consists of null vectors of a certain quadratic form and corresponds to the upper half of the ``light cone'' in the usual representation of $\SO(n,1)$; see \S\ref{section: null vectors} for more details. %$V=\R^{n+1}\setminus\{0\}$, which is similar to the result of Maucourant and Schapira for $n=2$: 

When $\Gamma$ is geometrically finite, the limit set of $\Gamma$, denoted $\Lambda(\Gamma) \subseteq \partial(\H^n),$ decomposes into radial and bounded parabolic limit points: $$\Lambda(\Gamma) = \Lambda_r(\Gamma) \sqcup \Lambda_{bp}(\Gamma).$$ For the precise definitions, see \S\ref{section: notation}.

\begin{proposition}\label{prop: vector asymptotic intro}
 Let $\Gamma$ be convex cocompact. For any $\ov\varphi \in C_c(V)$ and every $v \in V$ with $v^- \in \Lambda(\Gamma),$ as $T\to \infty$, we have that \[\frac{1}{T^{\delta_\Gamma/2}} \sum\limits_{\gamma \in \Gamma_T} \ov\varphi(v\gamma) \asymp \int_{V}\ov\varphi(u)\frac{d\ov\nu(u)}{(\norm{v}_2\norm{u}_2)^{\delta_\Gamma/2}},\] where the implied constant depends on $v$ and $\Gamma$. Here, $\delta_\Gamma$ denotes the critical exponent of $\Gamma$, $\|u\|_2$ denotes the Euclidean norm of $u \in \R^{n+1},$ and $\Gamma_T = \{\gamma \in \Gamma : \|\gamma\|\le T\}$, where $\|\gamma\|$ denotes the max norm of $\gamma$ as a matrix in $\SL_{n+1}(\R)$. The notation $v^- \in \Lambda_r(\Gamma)$ is discussed in \S\ref{section: applications}.
\end{proposition}

Here, the notation $a \asymp b$ means that there exists a constant $\lambda>1$ such that $$\lambda\inv \le \frac{a}{b}\le \lambda.$$ 

%The notation $v^-\in\Lambda(\Gamma)$ and 
The precise definition of the measure $\ov\nu$ is discussed in \S\ref{section: applications}. It is the pushforward of the measure $\nu$ defined in \S\ref{sec: BR defn}, which is part of the product structure of the \emph{Burger-Roblin (BR)} measure, defined fully in that section.

Let $U=\{u_\t : \t \in \R^{n-1}\}$ be the expanding horospherical subgroup for the frame flow $A$. Let $P\subset G$ be the parabolic subgroup which contains the contracting horospherical subgroup. Parametrizations of these groups are given in \S\ref{section: notation}. 

Proposition \ref{prop: vector asymptotic intro} is obtained by counting orbit points in $U\backslash G$. We will also establish a stronger version, specifically showing that a more precise ratio tends to 1. With additional assumptions on $\Gamma$, we obtain a quantitative version of this statement. We need to define additional notation in order to state this result.

Let $UAK$ be the Iwasawa decomposition of $\SL_{n+1}(\R)$, and let $\Psi: U \backslash G \to G$ be the map \[\Psi(Ug) = ak,\] where $g=uak$ in the Iwasawa decomposition.

%Let $\pi_U:G\to U\backslash G$ be the quotient map. 

We view $G$ as embedded in $\SL_{n+1}(\R)$. For $g \in G$, let $\|g\|$ denote the max norm as a matrix in $\SL_{n+1}(\R)$. The following ``product'' is useful for our statements (a similar definition exists in the $\SL_2(\R)$ case). For $x,y\in U\backslash G$, let 
\begin{equation}\label{eq: star}
x\star y:=\sqrt{\frac{1}{2}\norm{\Psi(x)\inv E_{1,n+1} \Psi(y)}},    
\end{equation}
where $E_{1,n+1}$ is the $(n+1)\times(n+1)$ matrix with one in the $(1,n+1)$-entry and zeros everywhere else. For $x \in U\backslash G$ and $g \in G$, $x \star xg$ measures the difference between the $U$ components of the Iwasawa decomposition of $x$ and $xg$. More specifically, it measures the $(1,n+1)$ component of $g$. 

For $L \subseteq G$, define $$L_T := \{g \in L : \norm{g}\le T\}$$ %When $h=e$, the identity in $G$, we simply write $L_T$.
and $$B_U(T):=\{u_\t \in U : \|\t\|\le T\},$$ where $\|\t\|$ denotes the max norm of $\t \in \R^{n-1}.$ Let $\pi_U : G \to U\backslash G$ denote the natural projection map.

We will be interested in the following quantity:
\begin{equation}\label{eq: defn of I varphi T x}
    I(\varphi,T,x):=\int_{P}\ps_{\Psi(x)\Gamma}\left(B_U\left(\frac{\sqrt{T}}{x\star\pi_U(p)}\right)\right){\varphi(\pi_U(p))}d\nu(p).
\end{equation} Here, $\varphi$ is a function on $U\backslash G$, $x\in U\backslash G$, $T > 0$, $\ps$ denotes the PS measure, fully defined in \S\ref{section: PS}, and $\nu$ is defined in \S\ref{sec: BR defn}.

For two functions of $T$, $a(T),b(T)$, we write \[a(T) \sim b(T) \iff \lim\limits_{T\to\infty}\frac{a(T)}{b(T)}=1.\] %if $\lim\limits_{T\to\infty} a(T)/b(T)=1.$

We can now state a qualitative version of our ratio theorem:

\begin{theorem}\label{thm: non effective ratio}
    Let $\Gamma$ be geometrically finite. For any $\varphi\in C_c(U\backslash G)$ and every $x \in U\backslash G$ such that $\Psi(x)^- \in \Lambda_r(\Gamma),$ \[\sum\limits_{\gamma \in \Gamma_T}\varphi(x\gamma)\sim I(\varphi,T,x).\]
    The notation $g^-$ for $g \in G$ is defined in \S\ref{section: notation}. 
\end{theorem}

%The assumption that all cusps of maximal rank is vacuously satisfied when $\Gamma$ is convex cocompact. This assumption is required because the behaviour of the PS measure is complicated and not well-understood in general. Our argument relies on knowing that the PS measure of a ball of radius $(1+\eta)T$ is close to the PS measure of a ball of radius $T$ (see Lemma \ref{lem: freindly used} for the precise statement we require). In general, this is not known. However, when the PS measure is \emph{absolutely friendly}, this follows from \cite[Corollary 9.15]{BR eqdistr}. Moreover, by \cite{friendly}, the PS measure is absolutely friendly if and only if all cusps have maximal rank, by \cite{friendly}. 

By the shadow lemma, Proposition \ref{prop:shadow lemma}, we obtain the following corollary, which will in turn imply Proposition \ref{prop: vector asymptotic intro}: %. In $U\backslash G$, this takes the following form:

%To better understand the above approximation, one can use the shadow lemma to deduce the following corollary in the convex cocompact case.

\begin{corollary}\label{cor: asymptotic}
    Assume that $\Gamma$ is convex cocompact. For any $\varphi\in C_c(U\backslash G)$ and every $x \in U\backslash G$ such that $\Psi(x)^-\in\Lambda(\Gamma),$ as $T \to \infty$, \[\frac{1}{T^{\delta_\Gamma/2}} \sum\limits_{\gamma \in \Gamma_T} \varphi(x\gamma) \asymp \int_{P}\frac{\varphi(\pi_U(p))}{(x\star\pi_U(p))^{\delta_\Gamma}}d\nu(p),\]
    where the implied constant depends on $x$ and $\Gamma$. 
\end{corollary}

\begin{remark}
The proof also works for $\Gamma$ geometrically finite when the geodesic of $\Psi(x)\Gamma$ is bounded. We must then assume that $\Psi(x)^- \in \Lambda_r(\Gamma).$
\end{remark}

In order to state the quantitative version of Theorem \ref{thm: non effective ratio}, we need an additional definition, which gives a precise formulation of the notion that $x \in G/\Gamma$ does not escape to the cusps ``too quickly'': 

\begin{definition}\label{def: Diophantine} For $0<\eps<1$ and $s_0\ge 1$, we say that $x\in G/\Gamma$ with $x^- \in \Lambda(\Gamma)$ is \textbf{$(\eps,s_0)$-Diophantine} if for all $s>s_0$,
	\begin{equation*}
    d(\mathcal{C}_0,a_{-s}x)<(1-\eps)s,
	\end{equation*} 
where $\mathcal{C}_0$ is a compact set arising from the thick-thin decomposition, and is fully defined in \S\ref{section: thick thin decomposition}.
We say that $x\in G/\Gamma$ is \textbf{$\eps$-Diophantine} if it is $(\eps,s_0)$-Diophantine for some $s_0$. 
\end{definition} 

\begin{remark}
A point $x\in G/\Gamma$ is $\eps$-Diophantine for some $\eps>0$ if and only if $x^-\in\Lambda_r(\Gamma)$, because Definition \ref{def: Diophantine} precisely says that $x^- \not\in \Lambda_{bp}(\Gamma)$, by the construction of the thick-thin decomposition.
\end{remark}

When $\Gamma$ is convex cocompact, every $x \in G/\Gamma$ with $x^- \in \Lambda(\Gamma)$ is $\eps$-Diophantine for some $\eps$, because all limit points are radial in this case. Observe also that in the lattice case, this condition is always satisfied, because $\Lambda(\Gamma) = \partial(\H^n)$. See \cite{BR eqdistr} for further discussion of this definition.

\begin{definition} \label{def: property A}
We say that $\Gamma$ satisfies \textbf{property A} if one of the following holds:
\begin{itemize}
    \item $\Gamma$ is convex cocompact, or
    \item $\Gamma$ is geometrically finite, and either \begin{enumerate}
        \item $n \le 4$ and $\H^n/\Gamma$ has a cusp of rank $n-1$, or 
        \item $\delta_{\Gamma}>n-2$.
    \end{enumerate} %and either $n=2,3$ and $\delta_\Gamma \ge \frac{n-1}{2}$, or $n\ge 4$ and $\delta_{\Gamma}>n-2$.
\end{itemize}
\end{definition}

%Note that property A may be replaced with the assumption that 

%See \S\ref{section: thick thin decomposition} for the definition of the rank of a cusp.

\begin{remark}
The assumptions on $\Gamma$ in Definition \ref{def: property A} are to ensure the effective equidistribution theorem in \cite[Theorem 1.4]{BR eqdistr} holds (see Theorem \ref{thm;br equidistr} for a statement of this theorem in this setting). As discussed in \cite{BR eqdistr}, this theorem holds whenever the frame flow satisfies an explicit exponential mixing statement, \cite[Assumption 1.1]{BR eqdistr}, and this condition is satisfied under the conditions in Definition \ref{def: property A}. However, Definition \ref{def: property A} could be replaced with assuming that the more technical statement \cite[Assumption 1.1]{BR eqdistr} is satisfied.% For Theorem \ref{thm;br equidistr}, Definition \ref{def: property A} can be replaced with assuming that \cite[Assumption 1.1]{BR eqdistr} holds; the maximal cusp rank assumption is not needed for this result.
\end{remark}

Throughout the paper, the notation $$x \ll y$$ means there exists a constant $c$ such that $$x \le cy.$$ If a subscript is denoted, e.g. $\ll_\Gamma$, this explicitly indicates that this constant depends on $\Gamma$.

\begin{theorem}\label{thm: main} Let $\Gamma$ satisfy property A. For any $0<\eps<1$, there exist $\ell=\ell(\Gamma)\in\N$ and $\kappa=\kappa(\Gamma,\eps)$ satisfying: for every $\varphi\in C^\infty_c(U\backslash G)$  and for every $x\in U\backslash G$ such that $\Psi(x)\Gamma$ is $\eps$-Diophantine, and for all $T\gg_{\Gamma,\supp{\varphi},x}1$, 
\begin{align*}
    &\left|\frac{\sum_{\gamma\in\Gamma_T}\varphi(x\gamma)}{\int_{P}\ps_{\Psi(x)\Gamma}\left(B_U\left(\frac{\sqrt{T}}{x\star \pi_U(p)}\right)\right){\varphi(\pi_U(p))}d\nu(p)}-1\right|\\
    &\ll_{\Gamma,\supp\varphi,x} T^{-\kappa}\left(1 + S_\ell(\varphi)\nu(\varphi\circ\pi_U)\inv\right).
    \end{align*}
\end{theorem}

The dependencies in this statement are quite explicit. The dependence of $T$ on $x$ in Theorem \ref{thm: main} arises from the constant in Lemma \ref{lem: bound on ug}, which is explicitly defined in that proof, and the precise Diophantine nature of $x$, through Theorem \ref{thm;br equidistr} (i.e. the $\eps$ and $s_0$ that appear in Definition \ref{def: Diophantine}). The implied dependence on $x$ in the conclusion is discussed at the end of section \S\ref{section: proof of main thm}.

If the support of the function is small enough, then we can get a more explicit estimate (see \S\ref{sec: proof of small support thm}). This is used as a main step in the proof of Theorem \ref{thm: main}. 

%For $x\in U\backslash G$ and a compact set $H\subset U\backslash G$, let $\mathcal{R}(H,x):=\max\limits_{y,z\in H}\frac{x\star y}{x\star z}.$

%\begin{theorem}\label{thm: main small support}
 %   Let $\Gamma$ satisfy property A. For any $0<\eps<1$, there exist $\ell=\ell(\Gamma)\in\N$ and $\kappa=\kappa(\Gamma,\eps)$ satisfying: for every $x\in U\backslash G$ such that $\Psi(x)\Gamma$ is $\eps$-Diophantine and every compact $\Omega\subset G$, there exists $T_0=T_0(x,\Omega)$ so that for every $T\ge T_0$, there exists $\eta=\eta(T,\ell,\kappa,n,\Omega)>0$ such that if $\varphi\in C^\infty_c(U\backslash G)$ with $\Psi(\supp\varphi)\subseteq \Omega$ and satisfies $\mathcal{R}(\supp\varphi,x)-1<\eta$, then for every $y\in\supp\varphi$,
  %  \begin{align*}
   %     &\left|\frac{1}{\ps_{\Psi(x)\Gamma}\left(B_U\left(\frac{\sqrt{T}}{x\star y}\right)\right)}\sum_{\gamma\in\Gamma_T}\varphi(x\gamma)-\int_{P}{\varphi(\pi_U(p))}d\nu(p)\right| \ll_{\Gamma,\Omega,x}S_\ell(\varphi)T^{-\kappa}.
%    \end{align*}
%\end{theorem}

This paper is organized as follows. In \S\ref{section: notation}, we present notation used throughout the paper, the definitions and fundamental properties of the measures we are working with, and the equidistribution theorems that will be key in our arguments. In \S\ref{section: G mod U}, we explore the duality between $\Gamma$ orbits on $U\backslash G$ and of $U$ orbits on $G/\Gamma$, and prove key lemmas that are common to the proofs of both Theorems \ref{thm: non effective ratio} and \ref{thm: main}. This involves a thickening argument, due to Ledrappier, to reduce the problem to that of equidistribution of $U$ orbits.  In \S\ref{section: proof of noneffective ratio}, we prove Theorem \ref{thm: non effective ratio}, using an equidistribution theorem of Mohammadi and Oh, Theorem \ref{thm: non effective eqdstr}. In \S\ref{section: proof of main thm}, we prove Theorem \ref{thm: main}, using a quantitative equidistribution theorem, Theorem \ref{thm;br equidistr}. %A key point in this proof is the \emph{global friendliness} of the PS measure when $\Gamma$ has all cusps of maximal rank; see Corollary \ref{cor: global friendly PS}. 
Finally, in \S\ref{section: applications}, we consider two specific examples, and prove Proposition \ref{prop: vector asymptotic intro}.\newline

\noindent\emph{Acknowledgements:} We would like to thank Barak Weiss for bringing this problem to our attention, and Amir Mohammadi for many useful discussions. We are grateful to the anonymous referee for their insightful comments on an earlier version of this manuscript, which significantly improved the paper. The first author was partially supported by the Eric and Wendy Schmidt Fund for Strategic Innovation.

\section{Notation and preliminary results in $G/\Gamma$}\label{section: notation}

Let $G= \SO(n,1)^\circ$ and let $\Gamma \subseteq G$ be a Zariski dense discrete subgroup. %Let $e$ denote the identity in $G$.
Let $\pi_\Gamma:G\to G/\Gamma$ be the quotient map. %Let $d$ be the right invariant Riemannian metric. %For $L \subseteq G$ and $h \in G$, define $$L_T(h) := \{g \in L : d(h,g)\le T\}.$$ When $h=e$, the identity in $G$, we will simply write $L_T$.

Let $\Lambda(\Gamma)\subseteq\partial(\H^n)$ denote the limit set of $G/\Gamma$, i.e., the set of all accumulation points of $\Gamma z$ for some $z\in\H^n\cup \partial(\H^n)$.

The \emph{convex core} of $X:=G/\Gamma$ is the image in $X$ of the minimal convex subset of $\H^n$ which contains all geodesics connecting any two points in $\Lambda(\Gamma)$.
	
We say that $\Gamma$ is \emph{geometrically finite} if a unit neighborhood of the convex core of $\Gamma$ has finite volume. 
	
Fix a reference point $o \in \H^n$. 
Let $K=\text{Stab}_G(o)$ and let $d$ denote the left $G$-invariant metric on $G$ which induces the hyperbolic metric on $K\backslash G=\mathbb{H}^n$. 
Fix $w_o \in \opt1(\H^n)$ and let $M = \stab_G(w_o)$ so that $\opt1(\H^n)$ may be identified with $M\backslash G$. For $w \in \opt1(\H^n)$, $$w^\pm \in \partial\H^n$$ denotes the forward and backward endpoints of the geodesic $w$ determines. For $g \in G$, we define $$g^\pm := w_o^\pm g.$$

We say that a limit point $\xi\in\Lambda(\Gamma)$ is \emph{radial} if there exists a compact subset of $X$ so that some (and hence every) geodesic ray toward $\xi$ has accumulation points in that set. We denote by $\Lambda_r (\Gamma)$ the set of all radial limit points. 

An element $g\in G$ is called \emph{parabolic} if the set of fixed points of $g$ in $\partial(\mathbb{H}^n)$ is a singleton. We say that a limit point is \emph{parabolic} if it is fixed by a parabolic element of $\Gamma$.  A parabolic limit point $\xi\in\Lambda(\Gamma)$ is called \emph{bounded} if the stabilizer $\Gamma_\xi$ acts cocompactly
	on $\Lambda(\Gamma)-{\xi}$. 
	
We denote by $\Lambda_r (\Gamma)$ and $\Lambda_{bp} (\Gamma)$ the set of all radial limit points and the set of all bounded parabolic limit points, respectively.
	Since $\Gamma$ is geometrically finite (see \cite{bowditch}),  \[\Lambda (\Gamma)=\Lambda_r (\Gamma)\cup\Lambda_{bp} (\Gamma).\]

Let $A=\left\{a_s\::\:s\in\R\right\}$ be a one parameter diagonalizable subgroup such that $M$ and $A$ commute, and such that the right $a_s$ action on $M\backslash G=\opt1(\H^n)$ corresponds to unit speed geodesic flow. %Let $U$ denote the expanding horospherical subgroup and $\tilde U$ be the contracting horospherical subgroup with respect to the forward direction of $a_s$. Let $P=MA\tilde U$ be the parabolic subgroup. 

We embed $G$ in $\SL_{n+1}(\R)$, parametrize $A$ by $A=\{a_s : s\in \R\}$, where\[a_s = \begin{pmatrix} e^{s} & & \\ & I & \\ & & e^{-s} \end{pmatrix}\] and $I$ denotes the $(n-1)\times(n-1)$ identity matrix, and let \[M=\left\{\begin{pmatrix} 1 & & \\ & m & \\ &&1\end{pmatrix}:m \in \SO(n-1)\right\}.\]  %\[a_s := \begin{pmatrix} e^{s/2} & 0 \\ 0 & e^{-s/2}\end{pmatrix}, s\in \R\] and \[
	%= \left\{\pm\begin{pmatrix}e^{i\theta} & 0\\ 0 & e^{-i\theta}\end{pmatrix}\::\:\theta\in\R\right\}. \]
	%\tcr{Do we need an explicit description of $M$?}
	
	%Let $X=G/\Gamma$ and 
	Let $U$ denote the expanding horospherical subgroup\[
	U=\left\{g\in G\::\:a_{-s}ga_{s}\rightarrow e\text{ as }s\rightarrow+\infty \right\},\]
	let $\tilde U$ be the contracting horospherical subgroup\[
	\tilde U=\left\{g\in G\::\:a_{s}ga_{-s}\rightarrow e\text{ as }s\rightarrow+\infty \right\},\]
	and let $P=MA\tilde U$ be the parabolic subgroup.
	
	The group $U$ is isomorphic to $\R^{n-1}$. We use the parametrization $U=\{u_\t  : \t \in \R^{n-1}\}$, where $\t$ is viewed as a row vector, and $$u_\t = \begin{pmatrix} 1 & \t & \frac{1}{2}\norm{\t}^2 \\ & I & \t^T \\ & & 1 \end{pmatrix}.$$ For more details on these parametrizations and the interactions between these groups, see \cite[\S 2]{BR eqdistr}. %Similarly, $\tilde U = \{v_\t$ $$v_\t =  \begin{pmatrix} 1 & & \\ \t^T & I & \\ \frac{1}{2}|\t|^2 & \t & 1  \end{pmatrix}.$$
	
	\subsection{Thick-thin Decomposition and the Shadow Lemma}\label{section: thick thin decomposition}

	There exists a finite set of $\Gamma$-representatives $\xi_1,\dots,\xi_q\in\Lambda_{bp}(\Gamma)$. For $i=1,\dots,q$, fix $g_i\in G$ such that $g_i^- =\xi_i$,
	and for any $R>0$, set
	\begin{equation}\label{eq:siegel sets}
	\calh_i(R):=\bigcup_{s>R} Ka_{-s} U g_i,\quad\mbox{and}\quad\mathcal{X}_i(R):=\calh_i(R)\Gamma
	\end{equation}
	(recall, $K=\text{Stab}_G (o)$). 
	Each $\calh_i(R)$ is a horoball of depth $R$.
	
	The \emph{rank} of $\mathcal{H}_i(R)$ is the rank of the finitely generated abelian subgroup $\Gamma_{\xi_i}=\operatorname{Stab}_\Gamma(\xi_i)$. It is known that each rank is strictly smaller than $2\delta_\Gamma$.
	
	Let 
	\begin{equation}\label{eq:BMS supp defn}
	    \supp\bms:=\left\{g\Gamma\in X\::\:g^{\pm}\in\Lambda(\Gamma)\right\}.
	\end{equation}
	Note that the condition $g^\pm \in \Lambda(\Gamma)$ is independent of the choice of representative of $x=g\Gamma$ in the above definition, because $\Lambda(\Gamma)$ is $\Gamma$-invariant. Thus, the notation $x^\pm \in \Lambda(\Gamma)$ is well-defined. For now, $\supp\bms$ is simply notation, but as we will see, this coincides with the support of the BMS measure, $\bms.$ We say that a point $x\in X$ is a \emph{BMS point} if $x\in\supp\bms$. 
	
	According to \cite{bowditch}, there exists $R_0\geq 1$ such that $\mathcal{X}_1(R_0),\dots,\mathcal{X}_q(R_0)$  are disjoint, and for some compact set $\mathcal{C}_0\subset X$, \[
\supp\bms\subseteq\mathcal{C}_0\sqcup\mathcal{X}_1(R_0)\sqcup\cdots\sqcup\mathcal{X}_q(R_0).\]

\subsection{Patterson-Sullivan Measure}\label{section: PS}

	A family of finite measures $\{\mu_x\::\:x\in\H^n\}$ on $\partial(\H^n)$ is called a \emph{$\Gamma$-invariant conformal density of dimension $\delta_\mu>0$} if for every $x,y\in\H^n$, $\xi\in\partial(\H^n)$ and $\gamma\in\Gamma$,\[
	\gamma_*\mu_x=\mu_{x\gamma}\:\text{ and }\:\frac{d\mu_y}{d\mu_x}(\xi)=e^{-\delta_\mu \beta_\xi(y,x)},\]
	where $\gamma_*\mu_x(F)=\mu_x(F\gamma)$ for any Borel subset $F$ of $\partial(\H^n)$. 
	
	We let $\{\nu_x\}_{x \in \H^n}$ denote the Patterson-Sullivan density on $\partial \H^n$, that is, the unique (up to scalar multiplication) conformal density of dimension $\delta_\Gamma$.
	
	For each $x\in\H^n$, we denote by $m_x$ the unique probability measure on $\partial(\H^n)$ which is invariant under the compact subgroup $\text{Stab}_G (x)$. Then $\left\{m_x\::\:x\in \H^n\right\}$ forms a $G$-invariant conformal density of dimension $n-1$, called the Lebesgue density.
	Fix $o \in \H^n$. 
	
	For $x, y \in \H^n$ and $\xi \in \partial(\H^n)$, the \emph{Busemann function} is given by \[ \beta_\xi(x,y):=\lim\limits_{t\to \infty} d(x,\xi_t) - d(y,\xi_t)\] where $\xi_t$ is a geodesic ray towards $\xi$.
	
	For $g\in G$, we can define measures on $Ug$ using the conformal densities defined previously. The Patterson-Sullivan measure (abbreviated as the PS-measure):
	\begin{equation}\label{eqn; defn of ps} 
	d\ps_{Ug}(u_\t g) := e^{\delta_\Gamma \beta_{(u_\t g)^+}(o, u_\t g(o))}d\nu_o((u_\t g)^+),
	\end{equation} 
	and the Lebesgue measure \[
	d\leb_{Ug}(u_\t g):=e^{(n-1) \beta_{(u_\t g)^+}(o, u_\t g(o))}dm_o((u_\t g)^+).\] 
	
	Note that for any $g\in G$, a point $h\in Ug$ satisfies $h\in\supp\ps_{Ug}$ if and only if $h^+\in\Lambda(\Gamma)$. Therefore, we refer to the points $x\in X$ which satisfy $x^+\in\Lambda(\Gamma)$ as \emph{PS points}. 
	
	The conformal properties of $m_x$ and $\nu_x$ imply that this definition is independent of the choice of $o\in\H^n$.
	
	We often view $\ps_{Ug}$ as a measure on $U$ via $$d\ps_{g}(\t):= d\ps_{Ug}(u_\t g).$$
	
	The measure	\[
	d\leb_{Ug}(u_\t g)=d\leb_U(u_\t)=d\t\]is independent of the orbit $Ug$
	and is simply the Lebesgue measure on $U\equiv\R^{n-1}$ up to a scalar multiple.
	
	If $x \in X$ is such that $x^- \in \Lambda_r(\Gamma)$, then $$u \mapsto ux$$ is injective, and we can define the PS measure on $Ux \subseteq X$, denoted $\ps_x$, simply by pushforward of $\ps_g$, where $x = g\Gamma$. In general, defining $\ps_x$ requires more care, see e.g. \cite[\S 2.3]{joinings} for more details. As before, we can view $\ps_x$ as a measure on $U$ via \[
	d\ps_x(\t)=d\ps_x(u_{\t}x).\]

%\begin{lemma}\label{lem; continuity x to psx} 
 %   The map $x \mapsto \ps_x$ is continuous, where the topology on the space of regular Borel measures on $U$ is given by $\mu_n \to \mu \iff \mu_n(f)\to\mu(f)$ for all $f\in C_c(U)$.
%\end{lemma}
%\begin{proof}
%	This is clear from the definition of the PS measure, since it is defined using the Busemann function and stereographic projection.
%\end{proof}

Recall that for $T>0$, \begin{equation} B_U(T) := \{u_\t : \|\t\| \le T\},\label{eq; defn of U balls} \end{equation} where $\|\t\|$ is the max norm of $\t$ as measured in $\R^{n-1}$.
%For a subset $S\subseteq\R^{n-1}$ and $\xi>0$, let \[\mathcal{N}(S,\xi)=\{u_\t\::\: \norm{\t-S}<\xi\}.\]

	We will need the following version of Sullivan's shadow lemma:

		\begin{proposition}[{\cite[Prop. 5.1, Remark 5.2]{MauSchap}}]\label{prop:shadow lemma}
		%Let $\H^n/\Gamma$ be a nonelementary geometrically finite hyperbolic surface.
		There exists a constant $\lambda=\lambda(\Gamma)\ge1$ such that for all $x \in \supp\bms$ and all $T>0$, we have \begin{align}\lambda\inv T^{\delta_\Gamma}e^{(k(x,T)-\delta_\Gamma)d(\mathcal{C}_0, a_{-\log T}x)}
		&\le \ps_x(B_U(T)) \\
		&\le \lambda T^{\delta_\Gamma}e^{(k(x,T)-\delta_\Gamma)d(\mathcal{C}_0\nonumber, a_{-\log T}x)},\nonumber\end{align} where $k(x,T)$ denotes the rank of the cusp containing $a_{-\log T}x$ (and is zero if $a_{-\log T}x \in \mathcal{C}_0$).\end{proposition}

	\begin{remark}
		In \cite{MauSchap}, the shadow lemma is proven using the distance measured in $\H^n/\Gamma.$ However, because $\mathcal{C}_0$ is $K$-invariant and $\H^n = K \backslash G$, we obtain the form above.
	\end{remark}
	
	\begin{remark}
	When $\Gamma$ is convex cocompact, $\mathcal{C}_0=\supp\bms$, and the shadow lemma simplifies to \[\lambda\inv T^{\delta_\Gamma}\le \ps_x(B_U(T)) \le \lambda T^{\delta_\Gamma}.\]
	\end{remark}
	
	We will need the following, which says that the PS measure is doubling.
	
	\begin{lemma}[{\cite[Corollary 9.9]{BR eqdistr}}]\label{lem: PS is doubling}
	There exist constants $\sigma_1 = \sigma_1(\Gamma)\ge \delta_\Gamma$, $\sigma_2 = \sigma_2(\Gamma)>0$ such that for every $c >0$, every $x \in \supp\bms$ and every $T>0$, \[\ps_x(B_U(cT))\ll_\Gamma \max\{c^{\sigma_1}, c^{\sigma_2}\}\ps_x(B_U(T)).\]	\end{lemma}

We will also require control of the PS measure of slightly larger balls, specifically as will be established below in Lemma \ref{lem: freindly used}. This will be a result of the friendliness of the PS density when $\Gamma$ is geometrically finite, established in \cite{friendly}. More specifically, we will show that the measure of the boundary of certain balls can be controlled.

Let $d$ be a left-invariant Riemannian metric on $G/\Gamma$ that projects to the hyperbolic distance on $\H^n$.

Denote by $d_E$ the Euclidean metric on $\R^{n-1}$. For a subset $S\subseteq\R^{n-1}$ and $\xi>0$, let \[
\mathcal{N}(S,\xi)=\{x\in \R^{n-1}\::\: d_E(x,S)\le\xi\}.\]
For $v\in\R^{n-1}$ and $r>0$, let \[
B(v,r)=\left\{u\in\R^{n-1}\::\:d_E(u,v)\le r\right\}\]
be the Euclidean ball of radius $r$ around $v$. 

We say that a hyperplane $L$ is \emph{on the boundary} of a closed ball $B$ if $$\emptyset\ne L\cap B \subseteq \partial(B).$$ Below, we obtain estimates for the PS measure of small neighbourhoods of hyperplanes on the boundary of a ball centered at a BMS point. Though not written here, estimates also hold when the center of the ball is a PS point but not a BMS point, as long as the ball is sufficiently small. In this case, one may use arguments similar to those in the appendix of \cite{BR eqdistr}.

We caution the reader that the estimates below hold only for hyperplanes on the boundary of such a ball; to obtain such estimates for general hyperplanes, \emph{absolute friendliness} of the PS density is necessary. By \cite[Theorem 1.9]{friendly}, this is satisfied if and only if all cusps of $\H^n/\Gamma$ are of maximal rank $n-1$ (note that this is vacuously satisfied if $\Gamma$ is convex cocompact). In this case, one may use \cite[Corollary 9.14]{BR eqdistr} when $\Gamma$ is geometrically finite or \cite[Theorem 2]{friendly 2} if $\Gamma$ is convex cocompact.

\begin{lemma}
\label{lem: new planarity density}
    There exists a constant $\alpha=\alpha(\Gamma)>0$ satisfying the following: for all $\lambda \in \Lambda(\Gamma)$, $\xi>0,$ $0<\eta\le 1$, and every hyperplane $L$ that is on the boundary of $B(\lambda,\eta)$, we have that \[\nu_o(\mathcal{N}(L,\xi)\cap B(\lambda,\eta))\ll_\Gamma \left(\frac{\xi}{\eta}\right)^\alpha \nu_o(B(\lambda,\eta)).\]
\end{lemma} \begin{proof}
    By \cite[Theorem 1.9]{friendly}, $\nu_o$ is \emph{friendly} when $\Gamma$ is geometrically finite. In particular, this means that there exists $\alpha=\alpha(\Gamma)>0$ such that for all $\lambda \in \Lambda(\Gamma)$, $\xi>0$, $0<\eta\le1$, and every affine hyperplane $L \subseteq \partial(\H^n)$,  $$\nu_o(\mathcal{N}(L,\xi\norm{d_L}_{\nu_o,B(\lambda,\eta)})\cap B(\lambda,\eta)) \ll_\Gamma {\xi}^\alpha \nu_o (B(\lambda,\eta)),$$ 
			where \[\norm{d_L}_{\nu_o,B(\lambda,\eta)}:=\sup\left\{d(\textbf{y},L)\::\:\textbf{y}\in B(\lambda,\eta)\cap \Lambda(\Gamma)\right\}. \]
			Since $\lambda \in \Lambda(\Gamma),$ for any $L$ that is on the boundary of $B(\lambda,\eta$), we have that \[\norm{d_L}_{\nu_o,B(\lambda,\eta)} \ge \eta/2.\] Thus, for any $L$ that is on the boundary of $B(\lambda,\eta)$,  we have \[\nu_o(\mathcal{N}(\xi\eta/2)\cap B(\lambda,\eta))\ll_\Gamma \xi^\alpha \nu_o(B(\lambda,\eta)).\] Replacing $\xi$ with $2\xi \eta\inv$ then implies that for every such $L$, \begin{equation*}
			\nu_o(\mathcal{N}(L,\xi)\cap B(\lambda,\eta))\ll_\Gamma \left(\frac{\xi}{\eta}\right)^\alpha \nu_o(B(\lambda,\eta)),\end{equation*} as desired.
\end{proof}

By flowing with $a_{-s}$ for $s>0$, we obtain similar estimates for large balls centered at BMS points:

\begin{corollary} \label{cor:new planarity BMS}Let $\alpha=\alpha(\Gamma)>0$ be as in Lemma \ref{lem: new planarity density}. For every $x\in\supp\bms$ such that $x^-\in\Lambda_r(\Gamma)$, every $\eta,\xi>0,$  and every hyperplane $L$ in the boundary of $B_U(\eta)x,$ we have \[\ps_x\left(\mathcal{N}_U(L,\xi)\cap B_U(\eta)\right) \ll_\Gamma \left(\frac{\xi}{\eta}\right)^\alpha \ps_x(B_U(\eta)).\]
\end{corollary}
\begin{proof}
    %Since the radial limit points are dense in the all the limit points, using the continuity of the PS-measure, we may assume that $x^- \in \Lambda_r(\Gamma)$.

    We will first prove that there exists a constant $c=c(\Gamma)>0$ so that for all $x \in \supp\bms\cap \mathcal{C}_0$ with $x^- \in \Lambda_r(\Gamma),$ $\xi>0$, $\eta$ satisfying $$0< \eta \le c\inv,$$ and every hyperplane $L$ in the boundary of $B_U(\eta)x$, the inequality in the statement is satisfied.
    
    Since $\mathcal{C}_0$ is compact, there exists $c=c(\Gamma)>1$ such that for any $x\in \mathcal{C}_0$ we can find $g\in G$ such that $x=g\Gamma$ and $d(o,g(o))<c$. Then, for any $u_\t \in B_U(c\inv)$, we have
	\begin{align*}
	    |\beta_{(u_\t g)^+}(o,u_\t g(o))| &\le d(u_{\t}\inv (o),g(o)) \\
	    &\le d(u_{-\t}(o),o)+d(o,g(o)) \\
	    &\le 2c.
	\end{align*}
    Then, by the definition of the PS measure, we have 
    \begin{align*}
	    \ps_x(\mathcal{N}_U(L,\xi)\cap B_U(\eta))
	    &=\ps_g(\mathcal{N}_U(L,\xi)\cap B_U(\eta))\\
	    &= \int_{\t \in \mathcal{N}_U(L,\xi)\cap B_U(\eta)} e^{\delta_{\Gamma}\beta_{(u_\t g)^+}(o,u_\t g(o))} d\nu_o((u_\t g)^+)\\
	    &\ll_\Gamma \nu_o\left(\operatorname{Pr}_{g^-}(\mathcal{N}(L,\xi)\cap B_U(\eta))\right), 
    \end{align*}
    where $\operatorname{Pr}_{g^-} : Ug \to \partial(\H^n)\setminus \{g^-\}$ is the visual map $w \mapsto w^+$.
    Using \cite[Corollary 9.5]{BR eqdistr}, we may assume that $c$ satisfies 
    \begin{equation*}\label{eq: br eqdist cor}
        \nu_o\left(\operatorname{Pr}_{g^-}(\mathcal{N}(L,\xi)\cap B_U(\eta))\right)\ll_{\Gamma}\nu_o\left(\mathcal{N}(L', c^2\xi)\cap B\left(g^+, c^2 \eta\right)\right), 
    \end{equation*}
    where $L'$ is a hyperplane in the boundary obtained from the projection of $L$ under the visual map $\operatorname{Pr}_{g^-}$. 
    Thus, for any $x\in \mathcal{C}_0$ and $u_\t \in B_U(c)$, we arrive at
    \begin{equation}\label{eq: ps measure to dens upper}
        \ps_x(\mathcal{N}_U(L,\xi)\cap B_U(\eta))\ll_{\Gamma}\nu_o\left(\mathcal{N}(L', c^2\xi)\cap B\left(g^+, c^2 \eta\right)\right),
    \end{equation}
    and in a similar way, one may also deduce 
    \begin{equation}\label{eq: ps measure to dens lower}
        \ps_x(B_U(\eta))\gg_{\Gamma}\nu_o\left(B\left(g^+, c^{-2} \eta\right)\right),
    \end{equation}
    
    Now, we may conclude
    \begin{align*}
	    \ps_x(\mathcal{N}_U(L,\xi)\cap B_U(\eta))
	    &\ll_{\Gamma}\nu_o\left(\mathcal{N}(L', c^2\xi)\cap B\left(g^+, c^2 \eta\right)\right) & \text{by \eqref{eq: ps measure to dens upper}} \\
			    &\ll_\Gamma \left(\frac{\xi}{\eta}\right)^\alpha \nu_o(B(g^+,c^2\eta))&\text{by Lemma \ref{lem: new planarity density}}\\
			    &\ll_\Gamma \left(\frac{\xi}{\eta}\right)^\alpha\nu_o(B(g^+,{c}^{-2} \eta)) &\text{by \cite[Lemma 9.6/9.7]{BR eqdistr}} \\
			    &\ll_\Gamma \left(\frac{\xi}{\eta}\right)^\alpha \ps_x(B_U(\eta))  & \text{by \eqref{eq: ps measure to dens lower}},
			\end{align*} 

			Now, let $x \in \supp\bms$ with $x^- \in \Lambda_r(\Gamma)$ and let $\eta>0$. Since $a_{-s}x$ has accumulation points in $\mathcal{C}_0$, there exists $s>0$ so that $e^{-s}\eta<c\inv$ and $a_{-s}x \in \mathcal{C}_0$. By the first step of the proof, we then have that
			\begin{align*}
			    \frac{\ps_x(\mathcal{N}_U(L,\xi)\cap B_U(\eta))}{\ps_x(B_U(\eta))}&=\frac{\ps_{a_{-s}x}(\mathcal{N}_U(L,e^{-s}\xi)\cap B_U(e^{-s}\eta))}{\ps_{a_{-s}x}(B_U(e^{-s}\eta))}\\
			    &\ll_\Gamma \left(\frac{e^{-s}\xi}{e^{-s}\eta}\right)^\alpha= \left(\frac{\xi}{\eta}\right)^\alpha.
			\end{align*}
		%	The result then follows for these $x$ by flowing. Density implies it holds for all $x \in \supp\bms.$
\end{proof}
					%	Now, consider $x\in\supp\bms\cap\mathcal{C}_0$ and a hyperplane $L$ in the boundary of $B_U(\eta)x$.  Since $x \in \mathcal{C}_0,$ $ce^d(o,\pi(x))$ is bounded by a constant $\tilde{c}>0$ depending only on $\Gamma$. If we then assume that $\eta \le \tilde{c}\inv,$ we have   
			\begin{proposition}\label{prop: boundary control BMS}
			     Let $\alpha=\alpha(\Gamma)>0$ be as in Corollary \ref{cor:new planarity BMS}. Then for all $x \in\supp\bms$ such that $x^-\in\Lambda_r(\Gamma)$, $T>0$, and $0<\eps \le 1$, we have that \[\ps_x(B_U((1+2\eps)T))-\ps_x(B_U(T))\ll_\Gamma \eps^\alpha \ps_x(B_U(T)).\]
			\end{proposition}
			\begin{proof}
			By the geometry of $(B_U((1+2\eps)T)-B_U(T))x$, there exists a constant $m$ depending only on $n$ and hyperplanes $L_1,\ldots, L_m$ in the boundary of $B_U((1+2\eps)T)x$ so that \[(B_U((1+2\eps)T)-B_U(T))x \subseteq\bigcup\limits_{i=1}^m \mathcal{N}_U(L_i,2\eps T) \cap B_U((1+2\eps)T)x.\] Then by Corollary \ref{cor:new planarity BMS}, we have that \begin{align*}
			    &\ps_x(B_U((1+2\eps)T))-\ps_x(B_U(T)) \\&\le \sum\limits_{i=1}^m \ps_x\left(\mathcal{N}_U(L_i,2\eps T)\cap B_U((1+2\eps)T)\right) \\
			    &\ll_\Gamma \left(\frac{2\eps T}{(1+2\eps)T}\right)^\alpha \ps_x(B_U((1+2\eps)T)) \\
			    &\ll_\Gamma \eps^{\alpha} \ps_x(B_U(T)) &\text{by Lemma \ref{lem: PS is doubling}}
			\end{align*} Note that the assumption $\eps\le 1$ is for convenience in the last step only; one may still use Lemma \ref{lem: PS is doubling} if $\eps$ is not bounded, but the exponent on $\eps$ must change.
			\end{proof}
			
		We can obtain estimates for all $(\eps,s_0)$-Diophantine points for balls that are sufficiently large (in a way that is uniform and linear in $s_0$). In fact, for any compact set $\Omega \subseteq G/\Gamma$, there exists a $T_0=T_0(\Omega)$ satisfying the statement below for all $x\in \Omega$ with $x^- \in \Lambda_r(\Gamma),$ see e.g. \cite[Lemma 3.3]{joinings}. Thus, the statement below could take many forms and this is not as strong as possible; we simply write it in a way that is useful for our setting.
			
		\begin{corollary} \label{cor:diophantine multiplicative}Let $\alpha = \alpha(\Gamma)>0$ be as in Proposition \ref{prop: boundary control BMS}, let $0<\eps \le 1$ and let $s_0\ge 1.$ There exists $T_0 = T_0(\Gamma,s_0)>0$ so that for every $(\eps,s_0)$-Diophantine point $x \in G/\Gamma$, all $T > 2T_0+1$, and all $\xi>0$,
		    %Let $x \in G/\Gamma$ be such that $x^- \in \Lambda_r(\Gamma).$  Then there exists $T_0=T_0(x)>0$ (the horodistance thing) so that for all $T > T_0+1$ and all $\eps>0,$ 
		    \begin{equation}\label{eq:boundary control statement}\ps_x(B_U((1+2\xi)T))-\ps_x(B_U(T)) \ll_\Gamma \left(\xi+\frac{T_0}{T-T_0}\right)^\alpha \ps_x(B_U(T)).\end{equation} In particular, if $x^- \in \Lambda_r(\Gamma),$ there exists $T_0=T_0(x)>0$ so that for all $T \ge 2T_0+1$ and all $\xi>0,$ \eqref{eq:boundary control statement} holds. 
		\end{corollary}
		\begin{proof} 
		    By \cite[Lemma 3.8]{BR eqdistr}, there exists $T_0=T_0(\Gamma,s_0)>0$ (in fact, it is linear in $s_0$) so that for every $(\eps,s_0)$-Diophantine point $x$, there exists  $$y \in B_U(T_0)x \cap \supp\bms.$$  For $T \ge T_0$, we have $$B_U(T-T_0)y\subseteq B_U(T)x\subseteq B_U(T+T_0)y.$$ In particular, \[B_U((1+2\xi)T)x \subseteq B_U((1+2\xi)(T+T_0))y \] and  \begin{equation}\label{eq:choice of T in boundary cor} B_U(T-T_0)y\subseteq B_U(T)x.\end{equation} 
		    
		    Since all Diophantine points are radial, by assuming that $T \ge 2T_0+1$, we may use Proposition \ref{prop: boundary control BMS} below:
		    \begin{align*}
		        &\ps_x(B_U(1+2\xi)T)-\ps_x(B_U(T))\\&\le\ps_y(B_U(1+2\xi)(T+T_0))-\ps_y(B_U(T-T_0))\\	  &\le\ps_y\left(B_U\left((1+2\xi)\left(1+\frac{2T_0}{T-T_0}\right)(T-T_0)\right)\right)-\ps_y(B_U(T-T_0))\\
		        &\ll_\Gamma\left(\xi+\frac{T_0}{T-T_0}+\frac{\xi T_0}{T-T_0}\right)^\alpha\ps_y(B_U(T-T_0))\text{ by Proposition \ref{prop: boundary control BMS}}\\
		        &\ll_\Gamma\left(\xi+\frac{(1+\xi)T_0}{T-T_0}\right)^\alpha\ps_x(B_U(T)) &\text{ by \eqref{eq:choice of T in boundary cor}}
		    \end{align*} 
		    Since $T \ge 2T_0+1,$ $$\frac{\xi T_0}{T-T_0} \le \xi,$$ and it can be absorbed into the $\xi$ term, completing the proof.
		\end{proof}
			
%	\begin{corollary}
%	    Let $\alpha = \alpha(\Gamma)>0$ be as in Lemma \ref{lem: new planarity density}, let $0<\eps \le 1$, and let $s_0\ge 1$. There exists $T_0 = T_0(\Gamma,s_0)>0$ so that for every $(\eps,s_0)$-Diophantine point $x \in G/\Gamma$, all $T\ge 2T_0+1,$ and all $\xi>0$, \begin{equation}
%	        \label{eq:new planarity additive} \frac{\ps_x(B_U(T+\xi))-\ps_x(B_U(T))}{\ps_x(B_U(T))} \ll_\Gamma \left(\frac{\xi}{T}\right)^\alpha.
%	    \end{equation} In particular, if $x^- \in \Lambda_r(\Gamma),$ there exists $T_0=T_0(x)>0$ so that for all $T \ge 2T_0+1$ and all $\xi>0,$ \eqref{eq:new planarity additive} holds.
%	\end{corollary}
%	\begin{proof}
%	    Let $T_0=T_0(\Gamma,s_0)$ be as in Corollary \ref{cor:diophantine multiplicative}. Then for any $x \in G/\Gamma$ that is $(\eps,s_0)$-Diophantine, 
%	\end{proof}

    We may now state the form in which we will need this control. The implied constant below depends on $x$ through the initial time in Corollary \ref{cor:diophantine multiplicative}, so it can be made uniform over a compact set or over all points with the same Diophantine properties. However, this level of detail is not necessary for our results.
	
	\begin{lemma}\label{lem: freindly used}
	Let $\alpha = \alpha(\Gamma)>0$ be as in Corollary \ref{cor:diophantine multiplicative}. For every $x \in G/\Gamma$ with $x^- \in \Lambda_r(\Gamma),$ $c>0$, $0<\eta \le 1,$ and $0<r_+<\ell<r_-$ satisfying 
    %Let $\Omega\subseteq G/\Gamma$ be a compact set, $x\in \Omega$ with $x^- \in \Lambda_r(\Gamma)$, $c>0$, $0<\eta\le 1$, and $0<r_{+}<\ell<r_{-}$ such that
    \[\frac{r_+}{r_-}<1+\eta,\]
   there exists $T_0=T_0(x,r_+,r_-)>0$ such that for any $T>T_0$,
\begin{align*}
    & \left|\ps_{x}\left(B_U\left(\frac{\sqrt{T}\pm c}{r_{\pm}}\pm\eta\right)\right)-\ps_{x}\left(B_U\left(\frac{\sqrt{T}}{\ell}\right)\right)\right|\nonumber\\
    &\ll_{\Gamma,x} \left(\eta+\frac{c+1}{\sqrt{T}}\right)^\alpha\ps_{x}\left(B_U\left(\frac{\sqrt{T}}{\ell}\right)\right)
\end{align*}
	\end{lemma}
	\begin{proof}
	First, observe that by Corollary \ref{cor:diophantine multiplicative}, there exists $T_1=T_1(x)$ so that for all $T \ge 2T_1+1$ and all $\xi>0$, \begin{equation}
	    \label{eq:using multiplicative to deduce additive} \frac{\ps_x(B_U(T+\xi))-\ps_x(B_U(T))}{\ps_x(B_U(T))} \ll_\Gamma \left(\xi + \frac{T_1}{T-T_1}\right)^\alpha\ll_\Gamma \left(\xi+\frac{T_1}{T}\right)^\alpha.
	\end{equation} This follows immediately from the fact that $$\ps_x(B_U(T+\xi)) \le \ps_x(B_U(1+2\xi)T).$$ 
	
	Thus, if we assume that $T$ is sufficiently large so that $\sqrt{T}/\ell \ge 2T_1+1$ (and note that this condition can be taken to rely on $r_-$ rather than on $\ell$ specifically), and note that by the assumption, \begin{equation*}
    1\le\frac{\ell}{r_{+}}\le 1+\eta, \label{eq: comparing star and R r}
\end{equation*} we see from \eqref{eq:using multiplicative to deduce additive} that
	\begin{align*}
& \ps_{x}\left(B_U\left(\frac{\sqrt{T}+c}{r_{+}}+\eta\right)\right)-\ps_{x}\left(B_U\left(\frac{\sqrt{T}}{\ell}\right)\right)\\
& \ll_{\Gamma}\left(\frac{r_{+}\inv (\sqrt{T}+c)+\eta-\ell\inv \sqrt{T}}{\ell\inv \sqrt{T}}+\frac{T_1}{\ell\inv\sqrt{T}}\right)^\alpha\ps_{x}\left(B_U\left(\frac{\sqrt{T}}{\ell}\right)\right)\\
&\ll_\Gamma \left(\frac{\ell r_{+}\inv (\sqrt{T}+c)+\ell\eta- \sqrt{T}+T_1}{\sqrt{T}}\right)^\alpha\ps_{x}\left(B_U\left(\frac{\sqrt{T}}{\ell}\right)\right)\\
&\ll_\Gamma \left(\frac{(1+\eta) (\sqrt{T}+c)+\ell\eta- \sqrt{T}+T_1}{\sqrt{T}}\right)^\alpha\ps_{x}\left(B_U\left(\frac{\sqrt{T}}{\ell}\right)\right)\\
&\ll_\Gamma \left(\eta+\frac{{c}+\eta\ell+T_1}{\sqrt{T}}\right)^\alpha\ps_{x}\left(B_U\left(\frac{\sqrt{T}}{\ell}\right)\right)\\
&\ll_{\Gamma,x}\left(\eta+\frac{{c}+\eta\ell+1}{\sqrt{T}}\right)^\alpha\ps_{x}\left(B_U\left(\frac{\sqrt{T}}{\ell}\right)\right)\end{align*}
Note that the implied constant depends on $x$ because we have absorbed the constant $T_1.$ Now, choose $T_0 \ge T_1$ so that $T \ge T_0$ implies $\frac{\ell}{\sqrt{T}}<1$ (a condition which depends on $x$ and $r_-$ in this case), which implies the claim because we may then absorb this term into the $\eta$ term. 

The second case can be shown in a similar way, with the choice of $T_0$ depending on $x$ and $r_+$ there. 
	\end{proof}

\subsection{Burger-Roblin and Bowen-Margulis-Sullivan Measures}\label{sec: BR defn}

Let $\pi:\opt1(\H^n)\rightarrow\H^n$ be the natural projection. 
	Recalling the fixed reference point $o \in \H^n$ as before, the map $$w \mapsto(w^+,w^-, s:= \beta_{w^-}(o,\pi(w)))$$ is a homeomorphism between $\opt1(\H^n)$ and $$(\partial(\H^n)\times\partial(\H^n) - \{(\xi,\xi):\xi \in \partial(\H^n)\})\times \R.$$
	
	This homeomorphism allows us to define the Bowen-Margulis-Sullivan (BMS) and the Burger-Roblin (BR) measure on $\opt1(\H^n)$, denoted by $\tbms$ and $\tbr$, respectively:  \[d\tbms(w):=e^{\delta_\Gamma \beta_{w^+}(o,\pi(w))}e^{\delta_\Gamma \beta_{w^-}(o,\pi(w))}d\nu_o(w^+)d\nu_o(w^-)ds,\]
	\[d\tbr(w):= e^{(n-1)\beta_{w^+}(o,\pi(w))}e^{\delta_\Gamma \beta_{w^-}(o,\pi(w))} dm_o(w^+)d\nu_o(w^-)ds.\]
	
	The conformal properties of $\left\{\nu_x\right\}$ and $\left\{m_x\right\}$ imply that these definitions are independent of the choice of $o\in\H^n$. Using the identification of $\opt1(\H^n)$ with $M\backslash G$, we lift the above measures to $G$ so that they are all invariant under $M$ from the left. By abuse of notation, we use the same notation ($\tbms$ and $\tbr$). These measures are left $\Gamma$-invariant, and hence induce locally finite Borel measures on $X$, which are the Bowen-Margulis-Sullivan measure $\bms$ and the Burger-Roblin measure $\br$, respectively.
	
	%The conformal properties of $\left\{\nu_x\right\}$ and $\left\{m_x\right\}$ imply that the above definition is independent of the choice of $o\in\H^n$. Using the identification of $\opt1(\H^n)$ with $M\backslash G$, we lift the above measure to $G$ so that it is invariant under $M$ from the left. By abuse of notation, we use the same notation ($\tbr$). This measure is left $\Gamma$-invariant, and hence induce locally finite Borel measures on $X$, which is the Burger-Roblin measure $\br$.
	
	Note that\[
	\supp\left(\br\right)=\left\{x\in X\::\:x^{-}\in\Lambda(\Gamma)\right\},	\]
    and the support of the BMS measure indeed satisfies \eqref{eq:BMS supp defn}. 
	
	Recall the definition of $\tilde{U}$, and $P = MA\tilde{U}$ from the begining of \S \ref{section: notation}. $P$ is the stabilizer of $w_o^+$ in $G$. Hence, one can define a measure $\nu$ on $Pg$ for $g \in G$, which will give us a product structure for $\tbms$ and $\tbr$ that will be useful in our approach. For any $g\in G$, define \begin{equation}\label{eq; defn of nu}
	d\nu(pg) := e^{\delta_\Gamma \beta_{(pg)^-}(o, pg(o))} d\nu_o(w_o^-pg)dmds,
	\end{equation} on $Pg$, where $s = \beta_{(pg)^-}(o,pg(o))$, $p = mav \in MA\tilde{U}$ and $dm$ is the probability Haar measure on $M$. 
	
	Then for any $\psi\in C_c(G)$ and $g\in G$, we have 
	\begin{equation}\label{eq:br product structure}
	\tbr(\psi)=\int_{Pg}\int_U \psi(u_\t pg)d\t d\nu(pg) 
	\end{equation} and 	\begin{equation}\label{eq:bms product structure}
	\tbms(\psi)=\int_{Pg}\int_U \psi(u_\t pg)d\ps_{pg}(\t)d\nu(pg).
	\end{equation}

\subsection{Sobolev Norms}

In the next section we formulate the equidistribution and effective equidistribution results which we will use in the proof of the main theorems. In order to formulate them, we first need to define Sobolev norms.  Our proofs will require constructing smooth indicator functions and partitions of unity with controlled Sobolev norms. This section also includes lemmas constructing such partitions.

%The error terms in our main theorems are in terms of Sobolev norms, which we define here. 

For $\ell\in\N$, $1\leq p\leq\infty$, and $\psi\in C^\infty(X)\cap L^p(X)$ we consider the following Sobolev norm\[
	S_{p,\ell}(\psi)=\sum\norm{U\psi}_p
	\]
where the sum is taken over all monomials $U$ in a fixed basis of $\mathfrak{g}=\mbox{Lie}(G)$ of order at most $\ell$, and $\norm{\cdot}_p$ denotes the $L^p(X)$-norm. Since we will be using $S_{2,\ell}$ most often, we set $$S_\ell=S_{2,\ell}.$$ 
	
%Our proofs will require constructing smooth indicator functions and partitions of unity with controlled Sobolev norms. We prove such lemmas below.

%\begin{lemma}[{\cite[Lemma 2.8]{BR eqdistr}}] \label{lem; existence of smooth indicator}    For every $\xi_1,\xi_2>0$ and $g \in G$, there exists a non-negative smooth function $\chi_{\xi_1,\xi_2}$ defined on $B_U(\xi_1+\xi_2)g$ such that $0 \le \chi_{\xi_1,\xi_2}\le 1 $, $S_\ell(\chi_{\xi_1,\xi_2}) \ll_{n,\Gamma} \xi_1^{n-1}\xi_2^{-\ell-(n-1)/2},$ and $$\chi_{\xi_1,\xi_2}(h) = \begin{cases} 0 & \text{ if } h \not\in B_U(\xi_1+\xi_2)g \\ 1 &\text{ if } h \in B_U(\xi_1-\xi_2)g\end{cases}.$$\end{lemma}

\begin{lemma}[{\cite[Lemma 2.4.7]{KleinbockMargulis}}]\label{lem:KleinbockMargulis}
$ $

\begin{enumerate}
    \item Let $X,Y$ be Riemannian manifolds, and let $\varphi \in C_c^\infty(X)$, $\psi \in C_c^\infty(Y)$. Consider $\varphi \cdot \psi$ as a function on $X \times Y$. Then \[S_\ell(\varphi\cdot \psi)\le c(X,Y)S_\ell(\varphi)S_\ell(\psi),\] where $c(X,Y)$ is a constant depending only on $X$ and $Y$ (independent of $\varphi,\psi$).
    \item Let $X$ be a Riemannian manifold of dimension $N$ and let $x \in X$. Then for any $0<r<1$, there exists a non-negative function $f\in C_c^\infty(X)$ such that $\supp(f)$ is contained in the ball of radius $r$ centered at $x$, $\int_X f = 1,$ and $$S_\ell(f)\le c(X,x)r^{-\ell+N/2},$$ where $c(X,x)$ is a constant depending only on $X$ and $x$, not $r$.
\end{enumerate}
\end{lemma}

The following lemma is an immediate consequence of the product rule.

\begin{lemma}\label{lem: sobolev product rule}
    Let $X$ be a Riemannian manifold and let $\varphi,\psi \in C_c^\infty(X)$. For any $\ell \in \N$, \[S_\ell(\varphi\cdot\psi)\ll_\ell S_\ell(\varphi)S_\ell(\psi).\]
\end{lemma}

\begin{lemma}\label{lem: sobolev chain rule}
    For any $\ell'$ there exists $\ell>\ell'$ which satisfies the following.
    Let $X,Y$ be Riemannian manifolds, $\varphi\in C_c^\infty(X)$, and $\psi:Y\rightarrow X$ be a smooth function. Then  \[S_{\ell'}(\varphi\circ\psi)\ll_{\ell',\psi} S_\ell(\varphi).\]
\end{lemma}

\begin{proof}
By the chain rule, for any $1\leq k\le \ell'$,
\begin{align*}
\norm{(\varphi\circ\psi)^{(k)}}_2 &\ll_{\psi,k}\sum_{i=0}^k\norm{\varphi^{(i)}\circ\psi}_2 \\
& \ll_{\psi,k}\sum_{i=0}^k\norm{\varphi^{(i)}}_{\infty}m^{\operatorname{Haar}}(\supp\varphi)\\ & 
\ll_{\psi,k} S_{\infty,\ell'}(\varphi)m^{\operatorname{Haar}}(\supp\varphi)\\
& \ll_{\psi,k} S_{\ell}(\varphi),
\end{align*} where in the last line, we have used \cite{sobolev} to choose $\ell>\ell'$ satisfying $$S_{\infty,\ell'}(f)m^{\operatorname{Haar}}(\supp f)\ll S_{\ell}(f)$$ for any $f$, where the implied constant is global.
\end{proof}

\begin{lemma}\label{lem:partition of unity}
		Let $H$ be a Riemannian manifold of dimension $N$, $0<r<1$, $\ell\in\N$, and $E$ a bounded subset of $H$. Then, there exists a partition of unity $\sigma_1,\dots,\sigma_k$ of $E$ in $H_r(E)=\{g \in G : d_H(g,E)\le r\}$ where $d_H$ denotes the Riemannian metric on $H$, i.e. \[
		\sum_{i=1}^k \sigma_i(x)=\begin{cases}
		0 & \text{if }x\notin H_r(E)\\
		1 & \text{if }x\in E,
		\end{cases}\]
		such that for some $u_1,\dots,u_k\in E$ and all $1\leq i\leq k$ \[
		\sigma_i\in C_c^\infty(H_r(u_i)),\quad S_\ell(\sigma_i)\ll_N r^{-\ell+N/2}. \] Moreover, \[\sum\limits_{i=1}^k S_\ell(\sigma_i) \ll_{N, E} r^{-\ell + N/2}.\]
	\end{lemma}
	
	\begin{proof}
	According to Lemma \ref{lem:KleinbockMargulis}(b) there exists a non-negative smooth function $\sigma$ supported on $H_{r/2}$  such that \[\int_{H}\sigma(h)dm^{\operatorname{Haar}}(h)=1,\quad S_\ell(\sigma)\ll_{N}r^{-\ell+N/2}. \]
		
	Since $H$ is a Riemannian manifold and $E$ is bounded, there exists a smooth partition of unity, $f_i:H\rightarrow\R$, $i=1,\dots,k$, such that each $f_i$ is supported on a ball of radius $r/2$ with a center $u_i\in E$ and \[
	\sum_{i=1}^k f_i(x)=\begin{cases}
	0 & \text{if }x\notin H_r(E)\\
	1 & \text{if }x\in H_{r/2}(E).
	\end{cases}	\]
	For $i=1,\dots,k$ define $\sigma_i$ by\[
	\sigma_i:=f_{i}*\sigma. \]
	We will show that $\sigma_1,\dots,\sigma_k$ satisfy the claim.
	
	By definition, for $i=1,\dots,k$, $\sigma_i$ is supported on a ball of radius $r$ and centered at a point in $E$. By Young's convolution inequality, we have
	\begin{equation}\label{eq: bound of convolution}
	    S_\ell(\sigma_i)\le S_{1,0}(f_i)S_\ell(\sigma)\ll_N r^{-\ell+N/2}.	
	\end{equation}
	
	For any $h\in E$, $h\inv E$ contains the identity, and so we have $h\inv H_{r/2}(E)\supseteq H_{r/2}$. Thus,
	\begin{align*}
	    \sum_{i=1}^k\sigma_i(h) & =\sum_{i=1}^k\int_{H}f_i\left(x\right)\sigma(hx\inv)dm^{\operatorname{Haar}}(x)\\
	    & =\int_{H}\sum_{i=1}^kf_i\left(x\right)\sigma\left(hx\inv\right)dm^{\operatorname{Haar}}(x)\\
	    & =\int_{H_{r}(E)}\sigma\left(hx\inv\right)dm^{\operatorname{Haar}}(x)\\
	    & =1. 
	\end{align*}
	
	If $h\notin H_{r} (E)$, then we have $h\inv  H_{r/2}(E)\cap H_{r/2}=\emptyset$. Hence, the above computation yields \[
	\sum_{i=1}^k\sigma_i(h)=0.\]
	
	Note that by \eqref{eq: bound of convolution}, and since $f_i$ is a partition of unity, we may also deduce
	\begin{align*}
	    \sum_{i=1}^k S_\ell(\sigma_i)& \le  S_\ell(\sigma)\sum_{i=1}^kS_{1,0}(f_i)\\ 
	    & =  S_\ell(\sigma)\int_{H}\sum_{i=1}^k f_i(x)dm^{\operatorname{Haar}}(x)\\ 
	    & \le  S_\ell(\sigma) m^{\operatorname{Haar}}(H_r(E))\\ 
	    & \ll_{N,E} r^{-\ell+N/2}.	
	\end{align*}
\end{proof}

\subsection{Equidistribution results}

For the proof of the main theorems we use the equidistribution results stated below. 

The following theorem was proved for $G=\SL_2(\R)$ by Maucourant and Schapira  in \cite{MauSchap} and for $G=\SO(n,1)^\circ$ by Mohammadi and Oh in \cite{joinings}.

\begin{theorem}\label{thm: non effective eqdstr} Let $\Gamma$ be geometrically finite. Fix $x\in G/\Gamma$ such that $x^-\in\Lambda_r(\Gamma)$. Then for any $\psi\in C_c(G/\Gamma)$ we have\[
\lim_{T\rightarrow\infty}\frac{1}{\ps_x(B_U(T))}\int_{B_U(T)}\psi(ux)du=\br(\psi).\]
\end{theorem}

%Recall the definition of $\mathcal{C}_0$ from \S\ref{section: thick thin decomposition}.

\begin{theorem}[{\cite[Theorem 1.4 and Remark 7.3]{BR eqdistr}}]
    Assume $\Gamma$ satisfies property A.
    For any $0<\eps<1$ and $s_0\ge 1$, there exist $\ell=\ell(\Gamma)\in\N$ and $\kappa=\kappa(\Gamma,\eps)>0$ satisfying: for every compact $\Omega\subset G/\Gamma$ and $\psi \in C_c^\infty(\Omega)$, there exists $c = c(\Gamma,\supp\psi)$ such that for every $x\in G/\Gamma$ that is $(\eps,s_0)$-Diophantine, and for all $r \gg_{\Gamma,\Omega,\eps} s_0$,\[\left|\frac{1}{\ps_x(B_U(r))} \int_{B_U(r)}\psi(u_\t x)d\t - \br(\psi)\right| \le c S_\ell(\psi)r^{-\kappa},\] where $S_\ell(\psi)$ is the $\ell$-Sobolev norm.\label{thm;br equidistr}
\end{theorem}

%\begin{remark}
%Note that if $\Gamma$ is not convex cocompact, i.e., there are cusps in $\H^n/\Gamma$, the assumption that every cusp of $\H^n/\Gamma$ has maximal rank implies that $$\delta_\Gamma>(n-1)/2.$$ 
%\end{remark}

\section{Duality between $G/\Gamma$ and $U\backslash G$}\label{section: G mod U}

The goal of this section is to prove the following proposition, which shows that one can use equidistribution results of $U$ orbits in $G/\Gamma$ in order to study the distribution of the points in $x\Gamma_T$ for $x\in U\backslash G$. This will be used to prove Theorems \ref{thm: non effective ratio} and \ref{thm: main}.

Recall that for $x,y\in U\backslash G$, we defined 
\begin{equation}\label{eq: star defn}
    x\star y:=\sqrt{\frac{1}{2}\norm{\Psi(x)\inv E_{1,n+1} \Psi(y)}},
\end{equation}
where $E_{1,n+1}$ is the $(n+1)\times(n+1)$ matrix with one in the $(1,n+1)$-entry and zeros everywhere else.

Recall the Iwasawa decomposition $G=\SO(n,1)^\circ = U \times A \times K$. Define a continuous section by $\Psi:U\backslash G\rightarrow AK$ by \[
\Psi(Ug)= ak,\]
where $g=uak$ is the Iwasawa decomposition of $g$. 

For $\varphi \in C_c(U\backslash G)$, define \begin{equation}\label{eq: defn of R and r first time}
R_\varphi:=\max\limits_{y \in \supp\varphi}(x\star y), \quad r_\varphi:=\min\limits_{y\in \supp\varphi}(x\star y).\end{equation}

\begin{proposition}\label{prop: setting up for equidistribution}
Let $\eta>0$, $\Omega\subset U\backslash G$ be a compact set, $\varphi \in C(\Omega)$, and $\psi\in C(B_U(\eta))$ be a non-negative function such that $\int_U\psi=1$. Fix $x \in U\backslash G.$ Define $F\in C_c(G/\Gamma)$ by  $$F(g\Gamma):=\sum_{\gamma\in\Gamma}\psi(u(g\gamma))\varphi(\pi_U(g\gamma)).$$ 
Then, for some $c=c(x,\Omega)>0$, \[\int_{B_U\left(\frac{\sqrt{T}-c}{R_\varphi}-\eta\right)} F(u_\t \Psi(x)\Gamma)d\t \le \sum\limits_{\gamma \in \Gamma_T} \varphi(x\gamma) \le \int_{B_U\left(\frac{\sqrt{T}+c}{r_\varphi}+\eta\right)} F(u_\t \Psi(x)\Gamma)d\t.\]
\end{proposition}

Observe that \[
g\Psi(Ug)^{-1}\in U\]
is by definition the $U$-component of the Iwasawa decomposition of $g$. 
Similarly, for any $g,h\in G$, the $U$-component of $\Psi(Uh)g$ is given by \[
\Psi(Uh)g\Psi(Uhg)^{-1}=(h\Psi(Uh)\inv)^{-1}(hg\Psi(Uhg)^{-1})\in U. \]
Hence, for any $x\in U\backslash G$ and $g\in G$, we can define $c_x(g):=\t\in\R^{n-1}$, where $\t$ is such that $\Psi(x)g\Psi(xg)\inv=u_\t$. 
%\begin{equation} \label{eq: c_h definition}    u_{c_x(g)}= \Psi(x) g \Psi(xg)\inv \in U.\end{equation} 
Then by the actions of $U$ and $A$, we can see that this satisfies 
\begin{equation}
    c_{e}(u_{\t}g)=c_{e}(g)+\t,\quad c_{e}(a_s g)=e^{s}c_e(g)
\end{equation}
and for any $x\in U\backslash G$,
\begin{equation}\label{eq: c_h vs c_e}
    c_{x}(g)=c_{e}(\Psi(x)g).
\end{equation} Observe that \eqref{eq: c_h vs c_e} implies that \begin{equation}\label{eq: c_e hg vs c_h g} {c_e(hg)} = {c_e(h)}+{c_e(\Psi(Uh)g)} = c_e(h) + c_{Uh}(g).\end{equation}

Note that for $g \in G$, $$g = u_{e}(g) \Psi(Ug).$$ That is, $u_{e}(g)$ is the $U$ component of the Iwasawa decomposition of $G$, and $\Psi(Ug)$ is the $AK$ component.
%Therefore, for any $g,h\in G$\begin{equation}\label{eq: split u(hg)}u_{c_e(hg)}=hg\Psi(Uhg)^{-1}=u_{c_e(h)}u_{c_h(g)}\end{equation}

%For $g, h \in G$, let \begin{equation}\label{eq: defn of u(g) notation}u(g):=u_{c_e(g)} \text{ and } u_h(g):=u_{c_h(g)}.\end{equation}

%Define \begin{equation}D(g) = \|e-g\|. \label{eq:defn of D}\end{equation}

For $x \in U\backslash G$ and $g \in G$, we will abuse notation and write $$u_x(g):= u_{c_x(g)}.$$

\begin{lemma}\label{lem: bound on ug}
    For any compact $\Omega\subset U\backslash G$ and $x\in U\backslash G$ there exist $c=c(\Omega,x)>0$ such that for any $xg \in\Omega$ and $T>c$, we have
    \begin{enumerate}
        \item If $\norm{g}\le T$, then 
        $u_x(g)\in B_U\left( \frac{\sqrt{T}+c}{x\star xg}\right)$. 
        \item If $\norm{g}\ge T$, then 
        $u_x(g)\not\in B_U\left(\frac{\sqrt{T}-c}{x\star xg}\right)$.
    \end{enumerate}
\end{lemma}

\begin{proof}
We have $g=\Psi(x)\inv u_x(g)\Psi(xg)$. For $\t:=c_x(g)$ we get
\begin{align*}
    g & =\Psi(x)\inv u_x(g)\Psi(xg)\\
    & =\Psi(x)\inv\left(I+ \begin{pmatrix} 0 & \t &  0\\ 0 & 0 & \t^T\\ 0 & 0 & 0\end{pmatrix}+\norm{\t}^2E_{1,n+1}\right)\Psi(xg)
\end{align*}
Denote 
\begin{align*}
    c_1 & :=\max_{y\in\Omega}\left\{\norm{\Psi(x)\inv\Psi(y)}\right\},\\
    %T_0 & :=\max\left\{\norm{\t}\::\:\forall g\in\Omega,\:    c_1 +\norm{\Psi(Uh)\inv\begin{pmatrix} 0 & \t &  0\\ 0 & 0 & \t^T\\ 0 & 0 & 0\end{pmatrix}\Psi(Uhg)}    \ge\norm{\t}^2  Uh\star Uhg    \right\} ,\\
    c_2 & :=\max_{y\in\Omega,\norm{\t}\le 1}\left\{\norm{\Psi(x)\inv \begin{pmatrix} 0 & \t &  0\\ 0 & 0 & \t^T\\ 0 & 0 & 0\end{pmatrix}\Psi(y)}\right\}. 
\end{align*}
Then, $c_1$ and $c_2$ are functions of $x$ and $\Omega$. 
By the triangle inequality,
\begin{align*}
    \norm{g}&\le \|\t\|^2 (x \star xg)^2+\|\Psi(x)\inv \Psi(xg)\| + \norm{\Psi(x)\inv \begin{pmatrix} 0 & \t & 0 \\ 0 & 0 & \t^T \\ 0 & 0 & 0 \end{pmatrix} \Psi(xg)}\\
    %&\le \|\Psi(Uh)\inv \Psi(Uhg)-e\| + \frac{\norm{\t}}{T_0}\max_{\norm{\t'}\le T_o}\left\{\|\Psi(Uh)\inv \begin{pmatrix} 0 & \t' & 0 \\ 0 & 0 & \t'^T \\ 0 & 0 & 0 \end{pmatrix} \Psi(Uhg)\|\right\} + \|\t\|^2 (Uh \star Uhg)\\
    &\le \|\t\|^2 (x \star xg)^2+c_1 + c_2\|\t\| . 
\end{align*}

In a similar way
\begin{align*}
    \norm{g}&\ge \|\t\|^2 (x \star xg)^2-\|\Psi(x)\inv \Psi(xg)\| - \norm{\Psi(x)\inv \begin{pmatrix} 0 & \t & 0 \\ 0 & 0 & \t^T \\ 0 & 0 & 0 \end{pmatrix} \Psi(xg)}\\
    &\ge \|\t\|^2(x\star xg)^2- c_1 - c_2\|\t\|. 
\end{align*}
We conclude that for any $g\in\Omega$,
\begin{align}\label{eq: bound on g and star}
    \left|\norm{g}-(x\star xg)^2 \norm{\t}^2\right|
    \le c_1+c_2\norm{\t}. 
\end{align}

Assume $\norm{g}\ge T\ge c_1$. Then, by \eqref{eq: bound on g and star}\[
0\le (x\star xg)^2 \norm{\t}^2+c_2\norm{\t}+(c_1-T).\]
Using the quadratic formula, we may deduce that the right hand side of the above equation is equal to zero when \[
\norm{\t}=\frac{-c_2\pm\sqrt{c_2^2+4(T-c_1)(x\star xg)^2}}{2(x\star xg)^2}\]
Since $(x\star xg)^2$ and $\norm{t}$ are non-negative, it follows that\[
\norm{\t}\ge\frac{-c_2+\sqrt{c_2^2+4(T-c_1)(x\star xg)^2}}{2(x\star xg)^2}\]
Using the inequality $\sqrt{a\pm b}\ge\sqrt{a}-\sqrt{b}$, we arrive at\[
\norm{\t}\ge\frac{\sqrt{T}}{x\star xg}-\frac{c_2+c_1x\star xg}{(x\star xg)^2}\]

A similar computation shows that $\norm{g}\le T$ implies \[
\norm{\t}\le\frac{\sqrt{T}}{x\star xg}+\frac{c_2+c_1x\star xg}{(x\star xg)^2}.\]

Letting $c$ be the maximum of $\frac{c_2+c_1x\star xg}{(x\star xg)^2}$ for $g \in \Omega$ completes the proof.
\end{proof}

\begin{lemma}\label{lem: go from phi to an integral} Let $\varphi \in C_c(U\backslash G)$ and suppose that $\psi\in C_c(U)$ satisfies $$\int_U \psi = 1.$$ For $g \in G$, define \[f(g)=\psi(u(g))\varphi(\pi_U(g)).\]
Then for every $g\in G$, \[\varphi(\pi_U(g)) = \int_{\supp(\psi)u(g)\inv} f(u_{\t}g)d\t.\]
\end{lemma}

\begin{proof}
By the definition of $\psi$,
    \begin{align*}
        \varphi(\pi_U(g))&=\varphi(\pi_U(g))\int_{\supp(\psi)}\psi(u_\t)d\t\\
        &=\varphi(\pi_U(g))\int_{u(g)^{-1}\supp(\psi)}\psi(u(g)u_\t)d\t.
    \end{align*}
    Since $\pi_U(u_{t}g)=\pi_U(g),$ we have %\tcr{I think the below is wrong. Why is $u_h(g)u_\t = u_h(u_{-\t}g)$? It was true for $u_e$, but unclear to me now. Maucourant-Schapira uses $u_e$ in the definition of $f$ actually. But then they seem to maybe have $u_h$ since they have $\Psi(u)$ everywhere?}
    \begin{align*}
        \varphi(\pi_U(g))&=\int_{u(g)\inv \supp(\psi)} \psi(u(u_{\t}g))\varphi(\pi_U(u_{\t}g))d\t\\
        &=\int_{u(g)^{-1}\supp(\psi)}f(u_{\t}g)d\t
    \end{align*}
\end{proof}

We are now ready to prove Proposition \ref{prop: setting up for equidistribution}. 

\begin{proof}[Proof of Proposition \ref{prop: setting up for equidistribution}] Without loss of generality, we may assume that $\varphi\ge 0$. Define $f:G\rightarrow\R$ by \[f(g)=\psi(u(g))\varphi(\pi_U(g)).\] By Lemma \ref{lem: go from phi to an integral}, for every $g \in G$, \begin{equation}\varphi(\pi_U(g))=\int_{u(g)\inv B_U(\eta)}f(u_\t g)d\t.\label{eq: noneffective use of varphi to integral}\end{equation}

By Lemma \ref{lem: bound on ug}, there exist $c>0$ depending on $\Omega$ and $x$ such that for all $T \ge c$, if $\gamma \in \Gamma_T$ and $x\gamma\in\Omega$, then \begin{align}
    u_x(\gamma)\inv B_U(\eta) &\subseteq  B_U\left(\frac{\sqrt{T}+c}{x\star x\gamma}+\eta\right) \label{eq: containment of u hgamma inv BUeta}.%\nonumber \\
    %&\subseteq B_U\left(2\sqrt{(T+R(h,\varphi))}+\eta+\|u(h)\|\right). 
\end{align}
Observe also that since $\supp(\psi)\subseteq B_U(\eta),$ if $u_\t \not\in u_x(\gamma)\inv B_U(\eta)$, then $$f(u_\t \Psi(x)\gamma)=\psi(u_\t u(\Psi(x)\gamma))\varphi(\pi_U(\Psi(x)\gamma))=0.$$
Thus, using \eqref{eq: c_e hg vs c_h g} and Lemma \ref{lem: go from phi to an integral}, for $\gamma \in \Gamma_T$ with $x\gamma\in\Omega$, we have that \begin{align} %\label{eq: varphi h gamma equal to integral over large ball}
    \varphi(x\gamma) &= \int_{u(\Psi(x)\gamma)\inv B_U(\eta)}f(u_{\t}\Psi(x)\gamma)d\t\nonumber\\
    &= \int_{u_x(\gamma)\inv B_U(\eta)}f(u_{\t}\Psi(x)\gamma)d\t\nonumber\\ 
    &=\int_{B_U\left(\frac{\sqrt{T}+c}{x\star x\gamma}+\eta\right)} f(u_{\t}\Psi(x)\gamma)d\t. \label{eq: integral proof large ball}
    %& \le\int_{B_U\left(\frac{\sqrt{T}+c}{r}+\eta\right)} f(u_{\t}\Psi(x)\gamma)d\t
\end{align}

Note that \begin{align}
& F(g\Gamma):=\sum_{\gamma\in\Gamma}f(g\gamma)\label{eq: defn of F g Gamma}\end{align}

Thus, from \eqref{eq: integral proof large ball}, for $r=r_\varphi:=\min\limits_{y\in \supp\varphi}(x\star y)$, we obtain
%By Theorem \ref{thm;br equidistr}, there exist $\ell,\kappa',c_2$ as in the statement of Theorem \ref{thm;br equidistr} and $T_1=T_1(x,\Omega)\ge c$ such that for $T \ge T_0$,
\begin{align}\sum_{\gamma\in\Gamma_T}\varphi(x\gamma) 
	&\le \sum\limits_{\gamma \in \Gamma_T} \int_{B_U\left(\frac{\sqrt{T}+c}{r}+\eta\right)}f(u_\t\Psi(x) \gamma)d\t \nonumber\\
	&\le\int_{B_U\left(\frac{\sqrt{T}+c}{r}+\eta\right)}F(u_\t \Psi(x)\Gamma)d\t. \nonumber
%	& \le \ps_{\Psi(x)\Gamma}\left(B_U\left(\frac{\sqrt{T}+c}{r}+\eta\right)\right) \left(\br(F)+c_2S_\ell(F)T^{-\kappa'}\right).\label{eq: upper bound of sum varphi gamma} 
\end{align}

To obtain a lower bound, we must control the terms arising from $\gamma \in \Gamma \setminus \Gamma_T$ in the definition of $F$. Note that by Lemma \ref{lem: bound on ug}, if $\gamma \in (\Gamma \setminus \Gamma_T)$ and $x\gamma\in\Omega$, then we see that \begin{equation*}\label{eq: uhgamma disjoint from smaller ball} u_x(\gamma)\inv B_U(\eta) \cap  B_U\left(\frac{\sqrt{T}-c}{x\star x\gamma}-\eta\right) =\emptyset.\end{equation*}Thus, by \eqref{eq: noneffective use of varphi to integral}, we obtain 
\begin{align*}
    \sum\limits_{\gamma \in \Gamma_T} \varphi(x\gamma) &= \sum\limits_{\gamma \in \Gamma_T} \int_{ B_U\left(\frac{\sqrt{T}-c}{x\star x\gamma}-\eta\right)} f(u_\t \Psi(x)\gamma)d\t\nonumber \\
    &=\sum\limits_{\gamma \in \Gamma} \int_{ B_U\left(\frac{\sqrt{T}-c}{x\star x\gamma}-\eta\right)} f(u_\t \Psi(x)\gamma)d\t\nonumber\end{align*} Now, similarly to the above, we conclude that \begin{align}
        \sum\limits_{\gamma \in \Gamma_T} \varphi(x\gamma) &\ge  \sum\limits_{\gamma \in \Gamma} \int_{ B_U\left(\frac{\sqrt{T}-c}{R}-\eta\right)} f(u_\t \Psi(x)\gamma)d\t\nonumber %\\
        %&= \int_{ B_U\left(\frac{\sqrt{T}-c}{R}-\eta\right)} F(u_\t \Psi(x)\Gamma)d\t\nonumber
       % &\ge \int_{ B_U\left(\frac{\sqrt{T}-c}{R}-\eta\right)} F(u_\t \Psi(x)\gamma)d\t\nonumber,
   % &\ge \int_{B_U\left(\sqrt{T}-c-\eta R\right)} F_2(u_\t \Psi(x)\Gamma)d\t \nonumber\\
    %&\ge \ps_{\Psi(x)\Gamma}\left(B_U\left(\sqrt{T}-c-\eta R\right)\right)\left(\br(F_2) - c_2S_\ell(F_2)T^{-\kappa'}\right)
    %&\ge  \ps_{\Psi(x)\Gamma}\left(B_U\left(\frac{\sqrt{T}+c}{R}-\eta\right)\right) \left(\br(F)+c_2S_\ell(F)T^{-\kappa'}\right).\label{eq: lower bound of sum varphi gamma}
\end{align} where $R=R_\varphi:=\max\limits_{y \in \supp\varphi}(x\star y).$ Then, the claim follows from the definition of $F$, \eqref{eq: defn of F g Gamma}.
\end{proof}

\begin{lemma}\label{lem: br computation}
Let $\varphi \in C_c(U\backslash G)$ and $F$ be as defined in Proposition \ref{prop: setting up for equidistribution}. Then, 
\begin{align}
    \br(F) =\int_{P}{\varphi(\pi_U(p))}d\nu(p)
\end{align} 
\end{lemma}

\begin{proof}
By the definition of $F$ and the assumption that $\int_U\psi=1$, by the product structure of the BR measure in \eqref{eq:br product structure}, we obtain
\begin{align*}
    \br(F)   
    & = \int_{G}\psi(u(g)){\varphi(\pi_U(g))}d\tbr(g)\\
    & =\int_{P}\int_U \psi(u_\t u(p)){\varphi(\pi_U(p))}d\t d\nu(p)\\
    & =\int_{P}{\varphi(\pi_U(p))}d\nu(p).
\end{align*} 
\end{proof}

For a set $H \subseteq G$, let \[B(H,r) = \{g \in G : d(g,H) \le r\},\] 
where $d$ is the Riemannian metric on $G$. That is, $B(H,r)$ is the $r$-thickening of $H$ with respect to $d$. For $h\in G$, we denote $B(\{h\},r)$ by $B(h,r)$ (in this case we get the Riemannian ball around the point $h$).

For $H \subseteq G$, denote by $$\inj(H)$$ the infimum over all $r>0$ satisfying that for every $h\in H,$ $$\pi_\Gamma|_{B(h,r)} : B(h,r) \to G/\Gamma$$ is injective.

In the later sections, we will require a partition of $\varphi$, say into $\varphi_1,\ldots,\varphi_k$ so that for each $i$, $R_{\varphi_i}$ and $r_{\varphi_i}$ are close. %This is so that the bounds obtained in Proposition \ref{prop: setting up for equidistribution} are not too far apart.

\begin{lemma}\label{lem: partition of the function}
Fix $x \in U\backslash G.$ For a compact set $H \subseteq G,$ there exists $0<\eta_0=\eta(H)<\inj(H),$ $\beta = \beta(H)>1$ so that for any $0<\eta<\eta_0$ and $\varphi \in C_c(U\backslash G)$ with $\supp\varphi\subset\pi_U(B(h,\eta))$ for some $h \in H$, we have that \[\frac{R_{\varphi}}{r_\varphi} -1 \le \|\Psi(x)\inv\|\beta\eta.\]
\end{lemma}

\begin{proof}
Since $B(H,1)$ is a compact set, by \cite[Lemma 9.12]{EinsiedlerWard}, there exist constants $0<\eta_0 = \eta(H)<\inj(H)$, $\beta=\beta(H)>1$, such that $\eta_0<1$ and for all $g,h \in B(H,1)$ with $d(g,h)\le \eta_0$, \begin{equation}
    \beta\inv \|g-h\|\le d(g,h) \le \beta \|g-h\|. \label{eq: equivalence of Riemannian and max}
\end{equation} 
Therefore, for any $h\in H$ and $0<\eta<\eta_0$, we have $$B(h,\eta) \subseteq \{g \in G : \|g-h\|\le \beta\eta\}.$$

Note that for any $g\in G$,\[
E_{1,n+1} \Psi(\pi_U(g))=E_{1,n+1} g.\]
Thus, if $\|g-h\|<\beta\eta$, then 
\begin{align*}
    \norm{\Psi(x)\inv E_{1,n+1} \Psi(\pi_U(g))}& =
    \norm{\Psi(x)\inv E_{1,n+1}g}\\
    &\le \norm{\Psi(x)\inv E_{1,n+1} h}+\norm{\Psi(x)\inv E_{1,n+1} (g-h)}\\
    &\le \norm{\Psi(x)\inv E_{1,n+1} \Psi(\pi_U(h))}+\beta\eta\norm{\Psi(x)\inv},
\end{align*} and similarly \[\norm{\Psi(x)\inv E_{1,n+1} \Psi(\pi_U(g))} \ge  \norm{\Psi(x)\inv E_{1,n+1} \Psi(\pi_U(h))} -\beta\eta\norm{\Psi(x)\inv}.\]

Thus, it follows from \eqref{eq: star defn} that for \begin{equation*}
    R = \max\limits_{y \in \pi_U(B(h,\eta))} (x\star y), \quad r = \min\limits_{y \in \pi_U(B(h,\eta))}(x\star y),
\end{equation*} we have $$R-r \le 2\beta \|\Psi(x)\inv\| \eta.$$
Since $r$ is bounded below by a constant depending on $H$, this implies that \[\left(\frac{R}{r}\right)-1 \ll_{H}\|\Psi(x)\inv\| \eta.\]
\end{proof}

\begin{corollary}\label{cor: partition of the function}
    Fix $x \in U\backslash G$ and $\varphi \in C_c(U\backslash G).$ Let $\eta_0=\eta_0(\Psi(\supp\varphi))$ be as in Lemma \ref{lem: partition of the function}. For any $0<\eta<\eta_0,$ there exist some $k$ and $\varphi_1,\ldots,\varphi_k \in C_c(U\backslash G)$ so that \[\sum\limits_{i=1}^k \varphi_i = \varphi \quad \text{ and } \frac{R_{\varphi_i}}{r_{\varphi_i}}-1 \ll_{\Gamma,\supp\varphi} \eta.\] Moreover, if $\varphi \in C_c^\infty(U\backslash G)$, then we also have $\varphi_i \in C_c^\infty(U\backslash G)$, and that for any $\ell'>0$, there exists $\ell > \ell'$ satisfying \begin{equation}
    \sum\limits_{i=1}^k S_{\ell'}(\varphi_i) \ll_{\ell,\supp\varphi} \eta^{-\ell+n(n+1)/4}S_\ell(\varphi).% \label{eq: sum of sobolev varphi i}
\end{equation}
\end{corollary}
\begin{proof}
For the first case (only assuming $\varphi \in C_c(U\backslash G)$), cover $\Psi(\supp\varphi)$ with balls of radius $\eta,$ and let $\sigma_1,\ldots,\sigma_k$ be a partition of unity subordinate to this cover. Defining $$\varphi_i = \varphi \cdot(\sigma_i \circ \Psi)$$ yields functions with the desired property, by Lemma \ref{lem: partition of the function}. 

Now, assume that $\varphi \in C_c^\infty(U\backslash G)$, and let $\ell>\ell'$ satisfy the conclusion of Lemma \ref{lem: sobolev chain rule} for $\ell'$. We must be more careful in order to control Sobolev norms. By Lemma \ref{lem:partition of unity}, for $0<\eta\le\eta_0$, there exist $h_1,\ldots, h_k \in \Psi(\supp\varphi)$ and $\sigma_1,\ldots,\sigma_k \in C_c^\infty(B(h_i,\eta))$ with \begin{equation}
    \sum\limits_{i=1}^k \sigma_i=1 \text{ on } \Psi(\supp\varphi) \text{ and } =0 \text{ outside } B(\Psi(\supp\varphi),\eta)
\end{equation}and such that \begin{equation}
   \sum\limits_{i=1}^k S_{\ell}(\sigma_i) \ll_{n,\supp\varphi} \eta^{-\ell + n(n+1)/4}. \label{eq: sum of sobolevs sigma i}
\end{equation}

Define $$\varphi_i = \varphi \cdot (\sigma_i\circ \Psi).$$
Then, by Lemma \ref{lem: partition of the function},\[
\frac{R_{\varphi_i}}{r_{\varphi_i}}-1\ll_{\supp\varphi}\|\Psi(x)\inv\|\eta.\]
Since $\Psi$ is smooth and $\ell>\ell'$, by Lemmas \ref{lem: sobolev product rule} and Lemma \ref{lem: sobolev chain rule}, \begin{align}
    S_{\ell'}(\varphi_i) &\ll_{\ell} S_{\ell'}(\varphi)S_{\ell'}(\sigma_i \circ \Psi) \nonumber\\
    &\ll_{\ell, \Psi} S_\ell(\varphi)S_\ell(\sigma_i).\label{eq: sobolev varphi i}
\end{align} 

From \eqref{eq: sum of sobolevs sigma i} and \eqref{eq: sobolev varphi i}, we conclude that
\begin{equation*}
    \sum\limits_{i=1}^k S_{\ell'}(\varphi_i) \ll_{\ell,n,\supp\varphi,\Psi} \eta^{-\ell+n(n+1)/4}S_\ell(\varphi).
\end{equation*}
\end{proof}

\section{Proof of Theorem \ref{thm: non effective ratio}}\label{section: proof of noneffective ratio}

In this section, we prove Theorem \ref{thm: non effective ratio}, which is restated below for convenience.

\begin{theorem}%\label{thm: non effective ratio}
    Let $\Gamma$ be geometrically finite. For any $\varphi\in C_c(U\backslash G)$ and every $x \in U\backslash G$ such that $\Psi(x)^- \in \Lambda_r(\Gamma),$ \[\sum\limits_{\gamma \in \Gamma_T}\varphi(x\gamma)\sim I(\varphi,T,x).\]
\end{theorem}

We will need the following lemma. Theorem \ref{thm: non effective ratio} will then follow by a partition of unity argument. %We will assume for the remainder of this section that all cusps have maximal rank.

\begin{lemma}\label{lem: small support}
Let $\varphi \in C_c(U\backslash G)$ and let $x\in U\backslash G$ be such that $\Psi(x)^-\in\Lambda_r(\Gamma)$. Let $R=R_\varphi$ and $r=r_\varphi$ be as in \eqref{eq: defn of R and r first time}. Let $\eta>0$, and suppose that $\frac{R}{r}<1+\eta$ and that $B_U(\eta)\Psi(\supp\varphi)$ injects into $G/\Gamma$.

Then for any $\eps>0$, there exists $T_1 = T_1(x,\eta,\varphi)>0$ such that for all $T \ge T_1$, \begin{align}
&\left|\sum\limits_{\gamma \in \Gamma_T} \varphi(x\gamma)-\int_{P}\ps_{\Psi(x)\Gamma}\left(B_U\left(\frac{\sqrt{T}}{x\star \pi_U(p)}\right)\right){\varphi(\pi_U(p))}d\nu(p)\right|\nonumber\\
&\ll_{\Gamma,x} \ps_{\Psi(x)\Gamma}\left(B_U\left(\frac{\sqrt{T}}{r}\right)\right)\left[\left(\eta+\frac{{c+1}}{\sqrt{T}}\right)^\alpha\int_{P}{\varphi(\pi_U(p))}d\nu(p) +\eps\right],\label{eq: non effestimate for nice varphi}
%&\ll_{h\Gamma,\supp\varphi} (\eta+T\inv)^{\alpha/2}\nu(\varphi\circ\pi_U)+\left(\ps_{h\Gamma}(B_U(\sqrt{2T}))\right)^{-1}S_\ell(\varphi)T^{-\kappa'}. 
\end{align} where $\alpha=\alpha(\Gamma)$ is from Lemma \ref{lem: freindly used}, and $c=c(x,\supp\varphi)>0$ is as in Proposition \ref{prop: setting up for equidistribution}. 
\end{lemma}
\begin{remark}
Note that $T_1$ depends on %$\varphi$\tcr{do we mean $\eta$?} 
$\eta$ through a non-canonical choice of bump function $\psi$, as seen in the proof. When we apply this lemma to a partition of unity, the same $\psi$ will be used for each part.
\end{remark}
\begin{proof}

Let $\psi \in C(B_U(\eta))$ be a non-negative function such that $\int_U \psi=1.$ Let $F$ and $c=c(x,\supp\varphi)>0$ be as in the statement of Proposition \ref{prop: setting up for equidistribution} for this $\psi$, and let $\eps>0$. 

By Theorem \ref{thm: non effective eqdstr}, there exists $T_1=T_1(x,\psi,\varphi)$ such that for $T \ge T_1$, \begin{align}
        &\ps_{\Psi(x)\Gamma}\left(B_U\left(\frac{\sqrt{T}-c}{R}-\eta\right)\right) \left(\br(F)-\eps\right)\label{eq: noneff lower bound of sum varphi gamma}\\
    &\le \sum\limits_{\gamma \in \Gamma_T} \varphi(x\gamma) \nonumber\\
    & \le \ps_{\Psi(x)\Gamma}\left(B_U\left(\frac{\sqrt{T}+c}{r}+\eta\right)\right) \left(\br(F)+\eps\right).\label{eq: non effupper bound of sum varphi gamma}
\end{align}

Let $y \in \supp(\varphi)$. By combining the above with Lemma \ref{lem: freindly used} (using $R, r$, and $\ell = x\star y$), we see that there exist constants $c_0=c_0(\Gamma,x)$ and $T_2 = T_2(\Gamma, x, \supp\varphi)>0$ such that for $T \ge T_2$, \begin{align}
   &\left(1-c_0\left(\eta+\frac{{c+1}}{\sqrt{T}}\right)^\alpha\right) \left(\br(F)-\eps\right)\nonumber\\
   &\le \frac{1}{\ps_{\Psi(x)\Gamma}\left(B_U\left(\frac{\sqrt{T}}{x\star y}\right)\right)}\sum\limits_{\gamma \in \Gamma_T} \varphi(x\gamma) \label{eq: noneff upper and lower bound after friendly}\\
   &\le \left(1+c_0\left(\eta+\frac{{c+1}}{\sqrt{T}}\right)^\alpha\right) \left(\br(F)+\eps\right).   \nonumber 
\end{align}

By Lemma \ref{lem: br computation}, $\br(F)=\int_P \varphi(\pi_U(p))d\nu(p)$, and so by \eqref{eq: noneff upper and lower bound after friendly}, for any $y\in\supp\varphi$, we obtain that
\begin{align*}
&\left|\frac{1}{\ps_{\Psi(x)\Gamma}\left(B_U\left(\frac{\sqrt{T}}{x\star y}\right)\right)}\sum\limits_{\gamma \in \Gamma_T} \varphi(x\gamma)-\int_{P}{\varphi(\pi_U(p))}d\nu(p)\right|\nonumber\\
&\ll_{\Gamma} \left(\eta+\frac{{c+1}}{\sqrt{T}}\right)^\alpha\int_{P}{\varphi(\pi_U(p))}d\nu(p) +\eps.%\label{eq: non effestimate for nice varphi}
%&\ll_{h\Gamma,\supp\varphi} (\eta+T\inv)^{\alpha/2}\nu(\varphi\circ\pi_U)+\left(\ps_{h\Gamma}(B_U(\sqrt{2T}))\right)^{-1}S_\ell(\varphi)T^{-\kappa'}. 
\end{align*}
Since the above holds for any $y \in \supp\varphi,$ by bounding
\begin{align*}
&\sum\limits_{\gamma \in \Gamma_T} \varphi(x\gamma)-\ps_{\Psi(x)\Gamma}\left(B_U\left(\frac{\sqrt{T}}{r}\right)\right)\int_{P}{\varphi(\pi_U(p))}d\nu(p). \nonumber\\
&\le\sum\limits_{\gamma \in \Gamma_T} \varphi(x\gamma)-\int_{P}\ps_{\Psi(x)\Gamma}\left(B_U\left(\frac{\sqrt{T}}{x\star \pi_U(p)}\right)\right){\varphi(\pi_U(p))}d\nu(p)\nonumber\\
&\le \sum\limits_{\gamma \in \Gamma_T} \varphi(x\gamma)-\ps_{\Psi(x)\Gamma}\left(B_U\left(\frac{\sqrt{T}}{R}\right)\right)\int_{P}{\varphi(\pi_U(p))}d\nu(p), 
\end{align*}
we obtain
\begin{align*}
&\left|\sum\limits_{\gamma \in \Gamma_T} \varphi(x\gamma)-\int_{P}\ps_{\Psi(x)\Gamma}\left(B_U\left(\frac{\sqrt{T}}{x\star \pi_U(p)}\right)\right){\varphi(\pi_U(p))}d\nu(p)\right|\nonumber\\
&\ll_{\Gamma,x} \ps_{\Psi(x)\Gamma}\left(B_U\left(\frac{\sqrt{T}}{r}\right)\right)\left[\left(\eta+\frac{{c+1}}{\sqrt{T}}\right)^\alpha\int_{P}{\varphi(\pi_U(p))}d\nu(p) +\eps\right].
\end{align*}
\end{proof}

We are now ready to prove Theorem \ref{thm: non effective ratio}.

\begin{proof}[Proof of Theorem \ref{thm: non effective ratio}]
%Fix $0<\eta<\inj(\Psi(\supp(\varphi)))$, where for $H \subseteq G$, $\inj(H)$ denotes the infimum over all $r>0$ such that \[\pi_\Gamma|_{B(g,r)} : B(g,r)\to G/\Gamma\] is injective, where $B(g,r)=\{h \in G : \|g-h\|\le r\}$.

By Corollary \ref{cor: partition of the function}, there exists $\eta_0=\eta_0(\Psi(\supp\varphi))>0$ so that for every $0<\eta<\eta_0,$ there exists $\left\{\varphi_i\::\:1\le i\le k\right\}$ that are a partition of $\varphi$, i.e.,\[
\varphi=\sum_{i=1}^k\varphi_i\] so that all the $ \varphi_i$ are supported on a small neighborhood of $\supp\varphi$, which we denote by $B$, and each $\varphi_i$ satisfies the assumptions of Lemma \ref{lem: small support}. %According to Corollary \ref{cor: partition of the function} we may assume further that for all $i$, $\varphi_i$ satisfies the assumptions of Lemma \ref{lem: small support}. For an explicit construction of such a partition, see \S \ref{section: proof of main thm}.

For any $1\le i\le k$ let, \[
R_i= R_{\varphi_i}, \quad r_i=r_{\varphi_i}\] as in \eqref{eq: defn of R and r first time}.

Note that \[
R:=\max_{y\in B}(x\star y),\quad r:=\min_{y\in B}(x\star y)\]
satisfy $R\ge R_i\ge r_i\ge r$ for any $i$. 

Fix $\eps>0$. By Lemma \ref{lem: small support}, there exists $T_1>0$ (depending on the $\varphi_i$'s, $x$, $\eta$, and $\eps$) such that for all $T \ge T_1$ and for each $i$, 
\begin{align*}
&\left|\sum\limits_{\gamma \in \Gamma_T} \varphi_i(x\gamma)-\int_{P}\ps_{\Psi(x)\Gamma}\left(B_U\left(\frac{\sqrt{T}}{x\star \pi_U(p)}\right)\right){\varphi_i(\pi_U(p))}d\nu(p)\right|\nonumber\\
&\ll_{\Gamma,x} \ps_{\Psi(x)\Gamma}\left(B_U\left(\frac{\sqrt{T}}{r}\right)\right)\left[\left(\eta+\frac{{c+1}}{\sqrt{T}}\right)^\alpha\int_{P}{\varphi_i(\pi_U(p))}d\nu(p) +\frac{\eps}{k}\right] .
\end{align*}

Summing over $i$, we obtain  
\begin{align}
&\left|\sum\limits_{\gamma \in \Gamma_T} \varphi(x\gamma)-\int_{P}\ps_{\Psi(x)\Gamma}\left(B_U\left(\frac{\sqrt{T}}{x\star \pi_U(p)}\right)\right){\varphi(\pi_U(p))}d\nu(p)\right|\nonumber\\
&\ll_{\Gamma,x} \ps_{\Psi(x)\Gamma}\left(B_U\left(\frac{\sqrt{T}}{r}\right)\right)\left[\left(\eta+\frac{c+1}{\sqrt{T}}\right)^\alpha\int_{P}{\varphi(\pi_U(p))}d\nu(p) +\eps\right] .\label{eq: ineffective nearly done}
\end{align}

Recall that \[I(\varphi, T, x):= \int_{P}\ps_{\Psi(x)\Gamma}\left(B_U\left(\frac{\sqrt{T}}{x\star \pi_U(p)}\right)\right){\varphi(\pi_U(p))}d\nu(p).\]

By Lemma \ref{lem: PS is doubling}, there exists $\sigma=\sigma(\Gamma)>0$ so that for any $y \in \supp\varphi$, \[\frac{\ps_{\Psi(x)\Gamma} \left(B_U\left(\frac{\sqrt{T}}{r}\right)\right)}{\ps_{\Psi(x)\Gamma} \left(B_U\left(\frac{\sqrt{T}}{x\star y}\right)\right)} \ll_\Gamma \left(\frac{R}{r}\right)^\sigma.\] Thus, from \eqref{eq: ineffective nearly done}, we obtain
\begin{align*}
    &\left|\frac{\sum\limits_{\gamma \in \Gamma_T} \varphi(x\gamma)}{I(\varphi,T,x)} -1 \right|\\ &\ll_\Gamma \left(\frac{R}{r}\right)^\sigma \nu(\varphi\circ \pi_U)\inv \left[\left(\eta+\frac{{c+1}}{\sqrt{T}}\right)^\alpha\int_{P}{\varphi(\pi_U(p))}d\nu(p) +\eps\right].
\end{align*}

Since $\eta$ and $\eps$ can be chosen arbitrarily small, the claim follows.
\end{proof}

We will now deduce Corollary \ref{cor: asymptotic} using the shadow lemma, Proposition \ref{prop:shadow lemma}.

\begin{proof}[Proof of Corollary \ref{cor: asymptotic}]

Since $\Psi(x)^- \in \Lambda_r(\Gamma)$, there exists $r=r(x)\ge 0$ such that $$B_U(r) \Psi(x)\Gamma\cap \supp\bms \ne \emptyset.$$ Let $w \in B_U(r)\Psi(x)\Gamma \cap \supp\bms \subseteq G/\Gamma.$ Then for any $T\ge 0$, \[\ps_w(B_U(T-r))\le\ps_{\Psi(x)\Gamma}(B_U(T))\le \ps_w(B_U(T+r)).\]

Thus, by Proposition \ref{prop:shadow lemma}, there exists $\lambda=\lambda(\Gamma)>1$ such that for all $T\ge 0$,\[\lambda\inv (T-r)^{\delta_\Gamma} \le \ps_{\Psi(x)\Gamma}(B_U(T)) \le \lambda (T+r)^{\delta_\Gamma}.\] %lambda\inv T^{\delta_\Gamma}\le\ps_w(B_U(T))\le \lambda T^{\delta_\Gamma}.\] 

For every $y \in \supp\varphi,$ we therefore have that for all $T \ge 2r$, \begin{equation}\label{eq: bound after using CC shadow lemma}\frac{T^{\delta_\Gamma/2}}{(x\star y)^{\delta_\Gamma}} \ll_{\Gamma, x} \ps_{\Psi(x)\Gamma}\left(\frac{\sqrt{T}}{x\star y}\right) \ll_{\Gamma,x}\frac{T^{\delta_\Gamma/2}}{(x\star y)^{\delta_\Gamma}}.\end{equation}

By Theorem \ref{thm: non effective ratio}, there exists $T_0=T_0(x,\varphi)$ such that for $T \ge T_0$, \[\left|\frac{\sum\limits_{\gamma \in \Gamma_T}\varphi(x\gamma)}{I(\varphi,T,x)} -1 \right| \le 1/2.\] Then \begin{align}
    \frac{1}{\ps_{\Psi(x)\Gamma}\left(\frac{\sqrt{T}}{x\star y}\right)}\sum\limits_{\gamma \in \Gamma_T} \varphi(x\gamma) &\le  \frac{2}{\ps_{\Psi(x)\Gamma}\left(\frac{\sqrt{T}}{x\star y}\right)} I(\varphi,T,x) \nonumber
\end{align} so by \eqref{eq: bound after using CC shadow lemma}, we obtain \begin{align}\frac{1}{T^{\delta_\Gamma/2}}\sum\limits_{\gamma \in \Gamma_T} \varphi(x\gamma) &\ll_{\Gamma,x} \frac{1}{T^{\delta_\Gamma/2}} \int_P \ps_{\Psi(x)\Gamma}\left(\frac{\sqrt{T}}{x\star \pi_U(p)}\right)\varphi(\pi_U(p))d\nu(p)\nonumber\\
&\ll_{\Gamma,x} \frac{1}{T^{\delta_\Gamma/2}} \int_P \frac{T^{\delta_\Gamma/2}}{(x\star \pi_U(p))^{\delta_\Gamma}}\varphi(\pi_U(p))d\nu(p)\nonumber\\
&\ll_{\Gamma,x} \int_P \frac{\varphi(\pi_U(p))}{(x\star \pi_U(p))^{\delta_\Gamma}}d\nu(p).\end{align} The lower bound is very similar.
\end{proof}

\section{A Small Support ``Ergodic Theorem''}\label{sec: proof of small support thm}

In this section, we prove an ergodic-theorem type statement for functions with small support. This result will be used in the next chapter to prove Theorem \ref{thm: main}. 

Recall that for $x\in U\backslash G$ and a compact set $H\subset U\backslash G$, let $$\mathcal{R}(H,x):=\max\limits_{y,z\in H}\frac{x\star y}{x\star z}.$$

\begin{theorem}\label{thm: effective rewrriten}
    Let $\Gamma$ satisfy property A. There exists $\ell=\ell(\Gamma) \in \N$ so that for any $0<\eps<1$, there exists $\kappa=\kappa(\Gamma,\eps)$ satisfying: for every $x\in U\backslash G$ such that $\Psi(x)\Gamma$ is $\eps$-Diophantine and every compact $\Omega\subset G$, there exists $T_0=T_0(x,\Omega)$ so that for every $T\ge T_0$, there exists $\eta=\eta(T,\ell,\kappa,n,\Omega)>0$ such that if $\varphi\in C^\infty_c(U\backslash G)$ with $\Psi(\supp\varphi)\subseteq \Omega$ and satisfies $\mathcal{R}(\supp\varphi,x)-1<\eta$, then for every $y\in\supp\varphi$,
    \begin{align*}
        &\left|\frac{1}{\ps_{\Psi(x)\Gamma}\left(B_U\left(\frac{\sqrt{T}}{x\star y}\right)\right)}\sum_{\gamma\in\Gamma_T}\varphi(x\gamma)-\int_{P}{\varphi(\pi_U(p))}d\nu(p)\right| \ll_{\Gamma,\Omega,x}S_\ell(\varphi)T^{-\kappa}.
    \end{align*}
\end{theorem}

\begin{proof}Fix $x \in U\backslash G$ such that $\Psi(x)\Gamma$ is $\eps$-Diophantine. Let $0<\eta_1=\eta_1(\Omega)<1$ be such that for all $g \in \Omega$, \[\pi_\Gamma|_{B(g,\eta_1)} : B(g,\eta_1)\to G/\Gamma\] is injective, where $B(g,\eta_1)=\{h \in G : \|g-h\|\le \eta_1\}$. Let $0<\eta<\eta_1.$ Then if $\Psi(\supp\varphi)\subset\Omega\subset G$, we have that \[B:=B_U(\eta)\Psi(\supp\varphi)\] injects into $G/\Gamma$. Let $R=R_\varphi$, $r=r_\varphi$ as in \eqref{eq: defn of R and r first time}. We are assuming that \begin{equation}\mathcal{R}(\supp\varphi,x) -1 = \frac{R}{r}-1<\eta.\label{eq: quotient of R,r}\end{equation} We will find $T_0=T_0(x,\Omega)$ as in the statement of the theorem, and choose $\eta$ depending on $T\ge T_0$ later.

According to Lemma \ref{lem:KleinbockMargulis}(2), there exists $\psi:U\rightarrow\R$ such that $\supp\psi=B_U(\eta)$ and 
\begin{equation}\label{eq: psi definition}
    \int_U \psi=1,\quad S_\ell(\psi)\ll\eta^{-\ell+n-1}. 
\end{equation}

We can now use Proposition \ref{prop: setting up for equidistribution} with the above $\psi$ and $\varphi$ to get an expression that we can estimate using the effective equidistribution theorem, Theorem \ref{thm;br equidistr}. 

Let $F$ and $c=c(\Omega,x)$ be as in Proposition \ref{prop: setting up for equidistribution} for $\psi,\varphi$. There exists $\ell,\kappa',c_2=c_2(\Gamma,\supp\psi,x)$ as in the statement of Theorem \ref{thm;br equidistr} and $T_1=T_1(x,\Omega)\ge c$ such that for all $T \ge T_1$,

\begin{align}
    &\ps_{\Psi(x)\Gamma}\left(B_U\left(\frac{\sqrt{T}+c}{R}-\eta\right)\right) \left(\br(F)-c_2S_\ell(F)T^{-\kappa'}\right)\label{eq: lower bound of sum varphi gamma}\\
    &\le \sum\limits_{\gamma \in \Gamma_T} \varphi(x\gamma) \nonumber\\
    & \le \ps_{\Psi(x)\Gamma}\left(B_U\left(\frac{\sqrt{T}+c}{r}+\eta\right)\right) \left(\br(F)+c_2S_\ell(F)T^{-\kappa'}\right).\label{eq: upper bound of sum varphi gamma}
\end{align}

We now need to express $\br(F)$ and $S_\ell(F)$ in terms of $\varphi$, and to compare the PS measures of the balls arising in \eqref{eq: lower bound of sum varphi gamma} and \eqref{eq: upper bound of sum varphi gamma}.

Let $y \in \supp\varphi.$ Note that, by definition of $r$ and $R$, $r\le x\star y\le R$. Hence, we may use Lemma \ref{lem: freindly used} to deduce that for $r_-:=R$ and $r_+:=r$, there exists $T_2=T_2(x,\Omega)>0$ so that for all $T\ge T_2$, we have that
\begin{align*}
    & \left|\ps_{\Psi(x)\Gamma}\left(B_U\left(\frac{\sqrt{T}\pm c}{r_{\pm}}\pm\eta\right)\right)-\ps_{\Psi(x)\Gamma}\left(B_U\left(\frac{\sqrt{T}}{x\star y}\right)\right)\right|\nonumber\\
    &\ll_{\Gamma,x} \left(\eta+\frac{{c+1}}{\sqrt{T}}\right)^\alpha\ps_{\Psi(x)\Gamma}\left(B_U\left(\frac{\sqrt{T}}{x\star y}\right)\right)
\end{align*}

According to Lemma \ref{lem: br computation}, we have 
\begin{align*}
    \br(F)   &=\int_{P}{\varphi(\pi_U(p))}d\nu(p).%\label{eq: BR measure of F bound}.
\end{align*} 

Combining the above with \eqref{eq: lower bound of sum varphi gamma} and \eqref{eq: upper bound of sum varphi gamma} implies that, for some $c_0=c_0(\Gamma,x)$, \begin{align}
   &\left(1-c_0\left(\eta+\frac{{c+1}}{\sqrt{T}}\right)^\alpha\right) \left(\int_{P}{\varphi(\pi_U(p))}d\nu(p)-c_2 S_\ell(F)T^{-\kappa'}\right)\nonumber\\
   &\le \frac{1}{\ps_{\Psi(x)\Gamma}\left(B_U\left(\frac{\sqrt{T}}{x\star y}\right)\right)}\sum\limits_{\gamma \in \Gamma_T} \varphi(x\gamma) \label{eq: upper and lower bound after friendly}\\
   &\le \left(1+c_0\left(\eta+\frac{{c+1}}{\sqrt{T}}\right)^\alpha\right) \left(\int_{P}{\varphi(\pi_U(p))}d\nu(p)+c_2 S_\ell(F)T^{-\kappa'}\right).   \nonumber 
\end{align}

We are left to find $S_\ell(F)$.
Since $B \mapsto B\Gamma$ is injective and $f$ is supported on $B$ (recall that $f$ is defined as in Lemma \ref{lem: go from phi to an integral}), using Lemma \ref{lem:KleinbockMargulis}(1), Lemma \ref{lem: sobolev chain rule}, and \eqref{eq: psi definition}, we have
\begin{align}
    S_{\ell}(F)& = S_{\ell}\left({f}\right)\nonumber\\
    & \ll_{n}  S_{\ell}(\psi)S_{\ell}\left(\varphi\circ\pi_U\right)\nonumber\\
    &\ll_{n,\Gamma}\eta^{-\ell+n-1}S_\ell(\varphi).
    %& \ll_{n,\Gamma}  \mu^{\operatorname{Haar}}(B_U(\eta-2\xi))^{-1}(\eta-\xi)^{n-1}\xi^{-\ell-(n-1)/2}S_{\ell}(\varphi)\nonumber\\
    %& \ll_{n,\Gamma}  \eta^{-\ell+(n-1)/2}S_{\ell}(\varphi).
    \label{eq: sobolev norm of F bound}
\end{align}

Finally, we need to put this all together. Combining  \eqref{eq: upper and lower bound after friendly} and \eqref{eq: sobolev norm of F bound}, for any $y\in\supp\varphi$, we obtain that
\begin{align}
&\left|\frac{1}{\ps_{\Psi(x)\Gamma}\left(B_U\left(\frac{\sqrt{T}}{x\star y}\right)\right)}\sum\limits_{\gamma \in \Gamma_T} \varphi(x\gamma)-\int_{P}{\varphi(\pi_U(p))}d\nu(p)\right|\nonumber\\
&\ll_{\Gamma,x}\left(\eta+\frac{{c+1}}{\sqrt{T}}\right)^\alpha\int_{P}{\varphi(\pi_U(p))}d\nu(p) +\eta^{-\ell+n-1}S_\ell(\varphi)T^{-\kappa'}\nonumber\\
&\ll_{\Gamma,\Omega,x} \left[\left(\eta+T^{-1/2}\right)^\alpha +\eta^{-\ell+n-1}T^{-\kappa'}\right] S_\ell(\varphi).\label{eq: estimate for nice varphi}
%&\ll_{h\Gamma,\supp\varphi} (\eta+T\inv)^{\alpha/2}\nu(\varphi\circ\pi_U)+\left(\ps_{h\Gamma}(B_U(\sqrt{2T}))\right)^{-1}S_\ell(\varphi)T^{-\kappa'}. 
\end{align}

Choose $\rho$ sufficiently small so that $$(\ell-n+1)\rho < \kappa'/2.$$ Let $\eta=T^{-\rho}$, for $T\ge T_0(x,\Omega):=\max\{T_1,T_2\}.$  Let $$\kappa=\min\{\rho\alpha,\alpha/2,\kappa'/2\}.$$ Then we conclude that

\begin{align}
    &\left|\frac{1}{\ps_{\Psi(x)\Gamma}\left(B_U\left(\frac{\sqrt{T}}{x\star y}\right)\right)}\sum\limits_{\gamma \in \Gamma_T} \varphi(x\gamma)-\int_{P}{\varphi(\pi_U(p))}d\nu(p)\right|\nonumber\\
    &\ll_{\Gamma,\Omega,x} T^{-\kappa}S_\ell(\varphi)\nonumber.
\end{align}\end{proof}

%%Since $\varphi$ is compactly supported and $\ell$ is chosen so that %$S_{\infty,1}(\varphi)\ll_{n}S_{\ell}(\varphi)$, we have \begin{equation} \label{eq: nu bound}
%    \nu(\varphi\circ\pi_U)\ll_{h\Gamma,\supp\varphi}S_{\ell}(\varphi).
%\end{equation}

%\begin{align}&\left|\sum\limits_{\gamma \in \Gamma_T} \varphi(x\gamma)-\int_{P}{\ps_{\Psi(x)\Gamma}\left(B_U\left(\frac{\sqrt{T}}{x\star \pi_U(p)}\right)\right)}{\varphi(\pi_U(p))}d\nu(p)\right|\nonumber\\&\ll_{\Gamma,\supp\varphi,x} \left[\frac{1}{\nu(\pi_U\inv(\supp\varphi))}\int_{P}{\ps_{\Psi(x)\Gamma}\left(B_U\left(\frac{\sqrt{T}}{x\star \pi_U(p)}\right)\right)\bold{1}_{\supp\varphi}(\pi_U(p))d\nu(p)}\right]\cdot\\&(1+\eta^\alpha)\left[\left(\eta+\frac{\sqrt{c}+\eta R}{\sqrt{T}}\right)^\alpha\int_{P}{\varphi(\pi_U(p))}d\nu(p) +\eta^{-\ell+(n-1)/2}S_\ell(\varphi)T^{-\kappa'}\right].\end{align}

\section{Proof of Theorem \ref{thm: main}}\label{section: proof of main thm}

In this section, we will use a partition of unity argument and the previous section to establish Theorem \ref{thm: main}, which is restated below for convenience. 

\begin{theorem} Let $\Gamma$ satisfy property A. There exists $\ell=\ell(\Gamma)\in \N$ so that for any $0<\eps<1$, there exists $\kappa=\kappa(\Gamma,\eps)$ satisfying: for every $\varphi\in C^\infty_c(U\backslash G)$  and for every $x\in U\backslash G$ such that $\Psi(x)\Gamma$ is $\eps$-Diophantine, and for all $T\gg_{\Gamma,\supp{\varphi},x}1$, 
\begin{align*}
    &\left|\frac{\sum_{\gamma\in\Gamma_T}\varphi(x\gamma)}{\int_{P}\ps_{\Psi(x)\Gamma}\left(B_U\left(\frac{\sqrt{T}}{x\star \pi_U(p)}\right)\right){\varphi(\pi_U(p))}d\nu(p)}-1\right|\\
    &\ll_{\Gamma,\supp\varphi,x} T^{-\kappa}\left(1 + S_\ell(\varphi)\nu(\varphi\circ\pi_U)\inv\right).
    \end{align*}
\end{theorem}

Assume throughout this section that $\Gamma$ satisfies property A. We begin by interpreting \eqref{eq: estimate for nice varphi} in another form, as in the following lemma. This form will be easier to work with when using a partition of unity. Note that the main idea here is that for $\varphi$ of small support and for any $y\in\supp\varphi$, $x\star y$ is very close to both $R$ and $r$.

For $H \subseteq U\backslash G$ compact and $x \in U\backslash G$, define $$R_H = \max\limits_{y \in H} x\star y \quad\text{ and } r_H = \min\limits_{y \in H} x\star y.$$ 

\begin{lemma} \label{lem: small varphi bring into integral} 
There exists $\ell=\ell(\Gamma)>0$ which satisfies the following. Let $\Omega \subseteq G$ be a compact set, let $x \in U \backslash G$ be such that $\Psi(x)\Gamma$ is $\eps$-Diophantine, let $\varphi\in C_c^\infty(U\backslash G)$ with $\Psi(\supp\varphi)\subset\Omega$, and let $\eta>0$ be smaller than the injectivity radius of $\Omega$. Let $R=R_{\pi_U(\Omega)}$ and $r=r_{\pi_U(\Omega)}$ and assume they satisfy $\frac{R}{r}-1<\eta.$ Then for $T \gg_{\Gamma,\Omega,x} 1,$
\begin{align*}
&\left|\sum\limits_{\gamma \in \Gamma_T} \varphi(x\gamma)-\int_{P}{\ps_{\Psi(x)\Gamma}\left(B_U\left(\frac{\sqrt{T}}{x\star \pi_U(p)}\right)\right)}{\varphi(\pi_U(p))}d\nu(p)\right|\nonumber\\
&\ll_{\Gamma,\Omega,x} {\ps_{\Psi(x)\Gamma}\left(B_U\left(\frac{\sqrt{T}}{r}\right)\right)}\left(\eta+T^{-1/2}\right)^\alpha\int_{P}{\varphi(\pi_U(p))}d\nu(p)\\ 
&\hspace{1cm} +{\ps_{\Psi(x)\Gamma}\left(B_U\left(\frac{\sqrt{T}}{r}\right)\right)}\eta^{-\ell+(n-1)/2}S_\ell(\varphi)T^{-\kappa'}.
\end{align*}
\end{lemma}

\begin{proof}
Following the arguments in the proof of Theorem \ref{thm: effective rewrriten} (more explicitely, the computations leading to \eqref{eq: estimate for nice varphi}), one may deduce that there exists $\ell=\ell(\Gamma)>0$ such that for any $T \gg_{\Gamma,\Omega,x} 1,$
\begin{align*}
&\left|\frac{1}{\ps_{\Psi(x)\Gamma}\left(B_U\left(\frac{\sqrt{T}}{x\star y}\right)\right)}\sum\limits_{\gamma \in \Gamma_T} \varphi(x\gamma)-\int_{P}{\varphi(\pi_U(p))}d\nu(p)\right|\nonumber\\
&\ll_{\Gamma,\Omega,x} \left[\left(\eta+T^{-1/2}\right)^\alpha +\eta^{-\ell+(n-1)/2}T^{-\kappa'}\right] S_\ell(\varphi).
\end{align*}
Therefore, we may conclude
\begin{align}
& -{\ps_{\Psi(x)\Gamma}\left(B_U\left(\frac{\sqrt{T}}{R}\right)\right)}\left[\left(\eta+T^{-1/2}\right)^\alpha\int_{P}{\varphi(\pi_U(p))}d\nu(p)\nonumber -\eta^{-\ell+(n-1)/2}S_\ell(\varphi)T^{-\kappa'}\right]\\
&\ll_{\Gamma,\supp\varphi,x} \sum\limits_{\gamma \in \Gamma_T} \varphi(x\gamma)-{\ps_{\Psi(x)\Gamma}\left(B_U\left(\frac{\sqrt{T}}{R}\right)\right)}\int_{P}{\varphi(\pi_U(p))}d\nu(p)\nonumber\\
&\le\sum\limits_{\gamma \in \Gamma_T} \varphi(x\gamma)-\int_{P}{\ps_{\Psi(x)\Gamma}\left(B_U\left(\frac{\sqrt{T}}{x\star \pi_U(p)}\right)\right)}{\varphi(\pi_U(p))}d\nu(p)\nonumber\\
&\le\sum\limits_{\gamma \in \Gamma_T} \varphi(x\gamma)-{\ps_{\Psi(x)\Gamma}\left(B_U\left(\frac{\sqrt{T}}{r}\right)\right)}\int_{P}{\varphi(\pi_U(p))}d\nu(p)\nonumber\\
&\ll_{\Gamma,\supp\varphi,x} {\ps_{\Psi(x)\Gamma}\left(B_U\left(\frac{\sqrt{T}}{r}\right)\right)}\left[\left(\eta+T^{-1/2}\right)^\alpha\int_{P}{\varphi(\pi_U(p))}d\nu(p) +\eta^{-\ell+(n-1)/2}S_\ell(\varphi)T^{-\kappa'}\right].\nonumber
\end{align}
\end{proof}

\bigskip

\begin{proof}[Proof of Theorem \ref{thm: main}]

\textbf{Step 1: Use an appropriate partition of $\varphi$.}

Let $\ell'=\ell'(\Gamma)>0$ satisfy the conclusion of Lemma \ref{lem: small varphi bring into integral} and $\ell>\ell'$ satisfy the conclusion of Corollary \ref{cor: partition of the function} for $\ell'$. 

By Corollary \ref{cor: partition of the function}, there exists a partition $\varphi_1,\ldots,\varphi_k$ of $\varphi$ satisfying Lemma \ref{lem: small varphi bring into integral} with $\Omega = \Psi(\supp\varphi)$ and \begin{equation}
    \sum\limits_{i=1}^k S_{\ell'}(\varphi_i) \ll_{\ell,\supp\varphi} \eta^{-\ell+n(n+1)/4}S_\ell(\varphi). \label{eq: sum of sobolev varphi i}
\end{equation} Thus, by Lemma \ref{lem: small varphi bring into integral}, we have that for each $\varphi_i$,
\begin{align}
&\left|\sum\limits_{\gamma \in \Gamma_T} \varphi_i(x\gamma)-\int_{P}\ps_{\Psi(x)\Gamma}\left(B_U\left(\frac{\sqrt{T}}{x\star\pi_U(p)}\right)\right){\varphi_i(\pi_U(p))}d\nu(p)\right|\nonumber\\
&\ll_{\Gamma,\supp\varphi,x} {\ps_{\Psi(x)\Gamma}\left(B_U\left(\frac{\sqrt{T}}{r_i}\right)\right)}\cdot\\
&\left[\left(\eta+T^{-1/2}\right)^\alpha\int_{P}{\varphi_i(\pi_U(p))}d\nu(p) +\eta^{-\ell+(n-1)/2}S_\ell(\varphi_i)T^{-\kappa'}\right].\nonumber
\end{align} 

Let \begin{equation*}
   r=\min\{r_1,\ldots,r_k\}.
\end{equation*} %Let $c>0$ satisfy the conclusion of Lemma \ref{lem: bound on ug} for $\Psi(x)\inv B(\Psi(\supp\varphi),1)$ and $x$. Then, by the proof of the Lemma \ref{lem: bound on ug}, it also satisfies the conclusion of Lemma \ref{lem: bound on ug} for $\{g \in G : \|g-h_i\|\le \beta\eta\}$ and $x$, i.e. we may replace $c_i$ with $c$ in the above. %and let $c$ be as in Lemma \ref{lem: bound on ug} for $\supp\varphi$ and $x$. (Note that $c\ge c_i$ for all $i$.)

Summing over $i$, using \eqref{eq: sum of sobolev varphi i}, and noting that $\eta<1$ yields
\begin{align}
&\left|\sum\limits_{\gamma \in \Gamma_T} \varphi(x\gamma)-\int_{P}\ps_{\Psi(x)\Gamma}\left(B_U\left(\frac{\sqrt{T}}{x\star\pi_U(p)}\right)\right){\varphi(\pi_U(p))}d\nu(p)\right|\nonumber\\
&\ll_{\Gamma,\supp\varphi,x} {\ps_{\Psi(x)\Gamma}\left(B_U\left(\frac{\sqrt{T}}{r}\right)\right)}
\left(\eta+T^{-1/2}\right)^\alpha\int_{P}{\varphi(\pi_U(p))}d\nu(p) \nonumber\\&\hspace{1cm}+{\ps_{\Psi(x)\Gamma}\left(B_U\left(\frac{\sqrt{T}}{r}\right)\right)}\eta^{-2\ell+(n^2+3n-2)/4}S_\ell(\varphi)T^{-\kappa'}.\label{eq: bound after partition of unity}
\end{align}

\bigskip
\textbf{Step 2: Putting it together.}

Recall \begin{equation*}
    I(\varphi,T,x):=\int_{P}\ps_{\Psi(x)\Gamma}\left(B_U\left(\frac{\sqrt{T}}{x\star\pi_U(p)}\right)\right){\varphi(\pi_U(p))}d\nu(p).
\end{equation*}

Let $$R=R_\varphi :=\max\limits_{y\in\supp\varphi} x\star y.$$ By Lemma \ref{lem: PS is doubling}, we have that there exists $\sigma=\sigma(\Gamma)>0$ so that  \begin{align}\frac{\ps_{\Psi(x)\Gamma}\left(B_U\left(\frac{\sqrt{T}}{r}\right)\right)}{I(\varphi,T,x)}
&\le \frac{\ps_{\Psi(x)\Gamma}\left(B_U\left(\frac{\sqrt{T}}{r}\right)\right)}{\ps_{\Psi(x)\Gamma}\left(B_U\left(\frac{\sqrt{T}}{R}\right)\right)\nu(\varphi\circ\pi_U)} \nonumber\\
&\ll_{\Gamma} \left(\frac{R}{r}\right)^\sigma \nu(\varphi\circ\pi_U)\inv \nonumber\\
&\ll_{\Gamma,\supp\varphi,x} \nu(\varphi\circ\pi_U)\inv,\label{eq: ratio with I after friendly}
\end{align} where the last line follows because $(R/r)^\sigma$ is simply a constant depending on $\supp\varphi,\Gamma,$ and $x$.

From \eqref{eq: bound after partition of unity} and \eqref{eq: ratio with I after friendly}, we obtain that%\[C:=\frac{1}{\int_{\pi_U\inv(\supp\varphi)\cap P}\ps_{\Psi(x)\Gamma}\left(B_U\left(\frac{\sqrt{T}}{x\star\pi_U(p)}\right)\right)d\nu(p)}\] and
\begin{align}
    \left|\frac{\sum\limits_{\gamma \in \Gamma_T} \varphi(x\gamma)}{I(\varphi,T,x)}-1\right|
    &\ll_{\Gamma,\supp\varphi,x} \frac{\ps_{\Psi(x)\Gamma}\left(B_U\left(\frac{\sqrt{T}}{r}\right)\right)}{I(\varphi,T,x)}
   \left(\eta+T^{-1/2}\right)^\alpha\int_{P}{\varphi(\pi_U(p))}d\nu(p)\nonumber\\
    &\hspace{2cm}+\frac{\ps_{\Psi(x)\Gamma}\left(B_U\left(\frac{\sqrt{T}}{r}\right)\right)}{I(\varphi,T,x)}\eta^{-2\ell+(n^2+3n-2)/4}S_\ell(\varphi)T^{-\kappa'}\nonumber\\
    &\ll_{\Gamma,\supp\varphi,x} 
   \sqrt{T}+c\left(\eta+T^{-1/2}\right)^\alpha+\nu(\varphi\circ\pi_U)\inv\eta^{-2\ell+(n^2+3n-2)/4}S_\ell(\varphi)T^{-\kappa'}\nonumber\\
    %&\ll_{\Gamma,\supp\varphi,x} \left(\frac{R}{r}\right)^\sigma \nu(\varphi\circ\pi_U)\inv \left[\left(\eta+\frac{\sqrt{c}+\eta R}{\sqrt{T}}\right)^\alpha \nu(\varphi\circ\pi_U) +\eta^{-2\ell+(n^2+3n-2)/4}S_\ell(\varphi)T^{-\kappa'}\right]\\
    &\ll_{\Gamma,\supp\varphi,x} \sqrt{T}+c\left(\eta+T^{-1/2}\right)^\alpha+ \frac{\eta^{-2\ell+(n^2+3n-2)/4}S_\ell(\varphi)T^{-\kappa'}}{\nu(\varphi\circ\pi_U)}\nonumber\\
    &\ll_{\Gamma,\supp\varphi,x} T^{-\kappa}\left(1 + S_\ell(\varphi)\nu(\varphi\circ\pi_U)\inv\right),\label{eq: after choosing eta in main thm proof}
\end{align} where \eqref{eq: after choosing eta in main thm proof} follows by choosing $\eta=T^{-\rho}$, where $\rho=1$ if $2\ell-\frac{n^2+3n-2}{4}<0$, and \[
\rho=\frac{\kappa'}{4\ell-n+1-\frac{1}{2}n(n+1)}\] otherwise, and letting \[
\kappa=\min\{\rho\alpha,\alpha/2,\kappa'/2\}.\]

\end{proof}
%\tcr{$\int_{P}\ps_{\Psi(x)\Gamma}\left(B_U\left(\frac{\sqrt{T}}{x\star\pi_U(p)}\right)\right){\varphi(\pi_U(p))}d\nu(p)$ to get a limit. What about\begin{align}
 %   &\frac{\int_{P}\ps_{\Psi(x)\Gamma}\left(B_U\left(\frac{\sqrt{T}}{x\star\pi_U(p)}\right)\right){\varphi(\pi_U(p))}d\nu(p)}{\int_{P\cap\pi_U\inv(\supp\varphi)} \ps_{\Psi(x)\Gamma}\left(B_U\left(\frac{\sqrt{T}}{x\star \pi_U(p)}\right)\right)d\nu(p)}?
%\end{align}}

\begin{remark} Note that the implied dependence on $x$ is quite explicit. It arises from suppressing the factors $R_\varphi, r_\varphi, \|\Psi(x)\inv\|,$ and $c$ throughout the argument. Specifically, $c$ is suppressed in the use of Lemma \ref{lem: small varphi bring into integral}, and $r_{\varphi},R_\varphi$ are suppressed in \eqref{eq: after choosing eta in main thm proof}. Note that these constants depend on $x$ and $\supp\varphi$ through the $\star$ operation, as can be seen from the definitions and the proof of Lemma \ref{lem: bound on ug}, and they can also be computed explicitly if desired. The factor of $\|\Psi(x)\inv\|$ is suppressed in the construction of the partition in Corollary \ref{cor: partition of the function}. %All of these may be computed explicitly (and bounded in ways that depend only on $\supp\varphi$ and $x$, not on $i$) from the definitions and the proof of Lemma \ref{lem: bound on ug}.
The implied constant from Theorem \ref{thm;br equidistr} also depends on $x$ through the explicit Diophantine behaviour of $x$, i.e. the $(\eps,s_0)$.\label{remark: dep on x}\end{remark}

\begin{remark}The suppressed constants $R_\varphi,r_\varphi, c,$ and $\|\Psi(x)\inv\|$ mentioned in Remark \ref{remark: dep on x} are continuous functions of $x$ by definition of $\star$. This will be used in the next section.\label{rem: continuous dep on x}\end{remark}

\section{Applications}\label{section: applications}

Let $V$ be a manifold on which $G$ acts smoothly and transitively from the right, so that $V$ may be identified with $H \backslash G$ for some closed subgroup $H$ of $G$ that is the stabilizer of a point $v_0\in V$. Let $\sigma : H \backslash G \to V$ be the identification \begin{equation} \label{eq: defn of sigma} \sigma(Hg) = v_0 \cdot g.\end{equation} Note that $\sigma$ is smooth because $G$ acts smoothly. 

Assume further that $U \subseteq H \subseteq UM$. %In particular, $H$ stabilizes $E_{1,n+1}$, which will be used later on to define $\star$ on $V$. 
In particular, $\pi_U(H)$ is compact in $U\backslash G$ (recall from \S\ref{section: notation} that $\pi_U : G \to U \backslash G$ is the quotient map). Define $\theta : U \backslash G \to H \backslash G$ by \begin{equation}\label{eq: defn of theta} \theta(Ug)=Hg.\end{equation}We will now show that $\theta$ is smooth. Since $U$ is closed, $\pi_U : G \to U\backslash G$ is a smooth submersion. Thus, $\theta$ is smooth if and only if $\theta \circ \pi_U$ is smooth. Since $\theta \circ \pi_U=\pi_H$, the quotient map from $G \to H \backslash G$, it is smooth, which establishes the smoothness of $\theta$. 

For $v, u \in V$, let $x,y \in U\backslash G$ be such that $u=\sigma(\theta(x)), v = \sigma(\theta(y))$. We may define \[v\star u = x\star y.\] This is well-defined because $UM$ stabilizes $E_{1,n+1}$, and $H \subseteq UM$ (see \eqref{eq: star} for the definition of $\star$ on $U\backslash G$).

Recall the definition of $\Psi:U\backslash G \to G$ from \S\ref{section: G mod U}: \[\Psi(Ug)=ak,\] where $g=uak$ is the Iwasawa decomposition of $g$.

\begin{definition}
A vector $v\in V$ is called \textbf{$\eps$-Diophantine} if there exists $x\in U\backslash G$ such that $v=v_0 \cdot x$ and $\Psi(x)\Gamma$ is $\eps$-Diophantine. Such $x$ is called an \textbf{$\eps$-Diophantine representative} of $v$.   
\end{definition}

\begin{remark}\label{rem: Diophantine doesnt depends on representative}
Note that for any $g \in G$, $g^-\in \Lambda_r(\Gamma)$ if and only if $(umg)^-\in\Lambda_r(\Gamma)$ for all $um \in UM$, since $UM$ does not change $g^-$. Thus, for $v \in V,$ we may define the notation $$v^- \in \Lambda_r(\Gamma)$$ if for any representative $\Psi(x),$ $\Psi(x)^-\in\Lambda_r(\Gamma).$ 
Note also that since $\mathcal{C}_0$ is $M$ invariant and $A$ commutes with $M$, the definition of $v$ being $\eps$-Diophantine is independent of the choice of a representative $x\in U\backslash G$. 
\end{remark}

%\tcr{Need to understand if there are actually many representatives for different $\eps,s_0$. Also want to understand what this condition means in specific $V$ settings.}

Observe that $\nu$ uniquely defines a measure on $U\backslash G$ by $\nu(\varphi\circ\pi_{U})$ for any continuous function $\varphi$ defined on $U\backslash G$. One can use the push-forward of this measure to $H\backslash G$ and the identification of $V$ with $H\backslash G$ to uniquely define a measure on $V$. Denote this measure by $\bar\nu$. 

\begin{corollary}\label{cor: main}
    For any $0<\eps<1$, there exist $\ell=\ell(\Gamma)\in\N$ and $\kappa=\kappa(\Gamma,\eps)$ satisfying: for every $\ov\varphi\in C^\infty_c(V)$ and $\eps$-Diophantine $v\in V$ with Diophantine representative $x\in U\backslash G$ (i.e., $v_0x=v$), and  $T\gg_{\Gamma,\supp{\ov\varphi},v}1$, 
    \begin{align*}
        &\left|\frac{\sum_{\gamma\in\Gamma_T}\ov\varphi(v\gamma)}{\int_{P}\ps_{\Psi(x)\Gamma}\left(B_U\left(\frac{\sqrt{T}}{v\star u}\right)\right){\ov\varphi(u)}d\ov\nu(u)}-1\right|\ll_{\Gamma,\supp\ov\varphi,x} T^{-\kappa}\left(1 + S_\ell(\ov\varphi)\nu(\ov\varphi)\inv\right).
    \end{align*}
\end{corollary}

\begin{proof} 
Let $\ell'$ satisfy the conclusion of Theorem \ref{thm: main} and $\ell$ satisfy the conclusion of Lemma \ref{lem: sobolev chain rule} for $\ell'$. 

Recall the definitions of $\sigma:H \backslash G \to V$ in \eqref{eq: defn of sigma} and $\theta: U \backslash G \to H \backslash G$ in \eqref{eq: defn of theta}. %Define $\varphi' \in C_c^\infty(H \backslash G)$ by \[\varphi'= \ov\varphi \circ \sigma.\]
Define $\varphi \in C_c^\infty(U \backslash G)$ by \[\varphi = \ov\varphi\circ\sigma\circ\theta.\] %Then, for $H=UM$, $\varphi'$ is ``$M$-invariant'' in this sense. We will identify $UM\backslash G$ with $\R^n \setminus\{0\}$ and interpret Theorem \ref{thm: main} in this setting.

%Theorem \ref{thm: main} for $\varphi'$ says that for any $h \in U\backslash G$ such that $\Psi(Uh)$ is $(\eps,s_0)$-Diophantine and $T$ sufficiently large, we have \[ \left|\frac{1}{\ps_{\Psi(Uh)\Gamma}\left(B_U(\sqrt{2T})\right)}\sum_{\gamma\in\Gamma_T}\varphi'(h\gamma)-\nu(\varphi'\circ\tau)\right|\ll_{\Gamma,\supp\varphi,d(e,\Psi(Uh)\Gamma)}S_\ell(\varphi')T^{-\kappa}.\]

%Let $x'\in H\backslash G$. Assume there exists $x\in U\backslash G$ such that $\theta(x)=x'$ and $\Psi(x)$ %(or $\Psi(\Theta(h))$) 
%is $(\eps,s_0)$-Diophantine. 
Let $x \in U\backslash G$ be an $\eps$-Diophantine representative of $v$. In particular, note that $\sigma(\theta(x))=\sigma(H \Psi(x)) = v.$ Then, since $$\varphi(x\gamma) = \ov\varphi(\sigma(\theta(x))\cdot\gamma) = \ov\varphi(v\cdot\gamma),$$ by Theorem \ref{thm: main}, 
for $T\gg_{\Gamma,\supp\ov\varphi,\eps,x} 1$, \begin{align*}
        & T^{-\kappa}\left(1 + S_\ell(\varphi)\nu(\varphi\circ\pi_U)\inv\right)\\
        &\gg_{\Gamma,\supp\varphi,x} T^{-\kappa}\left|\frac{\sum_{\gamma\in\Gamma_T}\varphi(x\gamma)}{\int_{P}\ps_{\Psi(x)\Gamma}\left(B_U\left(\frac{\sqrt{T}}{x\star \pi_U(p)}\right)\right){\varphi(\pi_U(p))}d\nu(p)}-1\right|\\
        &\gg_{\Gamma,\supp\ov\varphi,x} T^{-\kappa}\left|\frac{\sum_{\gamma\in\Gamma_T}\ov\varphi(v\gamma)}{\int_{P}\ps_{\Psi(x)\Gamma}\left(B_U\left(\frac{\sqrt{T}}{v\star u}\right)\right){\ov\varphi(u)}d\ov\nu(u)}-1\right|.
    \end{align*}

Note that the dependence of $T$ on $x$ is through $\eps,s_0$ such that $x$ is $(\eps,s_0)$-Diophantine, and by Remark \ref{rem: Diophantine doesnt depends on representative}, this is in fact independent of the choice of Diophantine representative $x$ of $v$. By Remark \ref{rem: continuous dep on x}, the dependence on $x$ in the implied constant in the above inequality can be made uniform over all representatives of $v$, as they vary by elements in $M$, a compact set. Thus, both dependencies on $x$ can be replaced by dependence on $v$.

Observe that $\varphi$ can be viewed as a function on $U\backslash H \times H\backslash G\cong U\backslash G$ by \[
\varphi(y,x)=\operatorname{id}_{U\backslash H}(y)\cdot(\ov\varphi\circ\sigma)(x). \]
Therefore, Lemma \ref{lem:KleinbockMargulis} and Lemma \ref{lem: sobolev chain rule} imply \[
S_{\ell'}(\varphi)\ll_{H}S_{\ell'}(\operatorname{id}_{U\backslash H})S_{\ell'}(\ov\varphi\circ\sigma)\ll_{H,\sigma,\supp\varphi}S_\ell(\ov\varphi),\]
where the Sobolev norm of $\operatorname{id}_{U\backslash H}$ is finite since we are assuming $U\backslash H$ is compact. 
\end{proof}

In a similar way, one may deduce the following from Corollary \ref{cor: asymptotic} (see Remark \ref{rem: Diophantine doesnt depends on representative} for the notation $v^-\in\Lambda_r(\Gamma)$):

\begin{corollary}\label{cor: asymptotic in general application}
    Assume that $\Gamma$ is convex cocompact. For any $\ov\varphi\in C_c(V)$ and every $v\in V$ with $v^-\in\Lambda_r(\Gamma)$, as $T \to \infty$, \[\frac{1}{T^{\delta_\Gamma/2}} \sum\limits_{\gamma \in \Gamma_T} \ov\varphi(v\gamma) \asymp \int_{P}\frac{\ov\varphi(u)}{(v\star u)^{\delta_\Gamma}}d\ov\nu(u),\]
    where the implied constant depends on $v$ and $\Gamma$. 
\end{corollary}

\subsection{Identification with null vectors}\label{section: null vectors}

Let $G$ act on $\R^{n+1}$ by right matrix multiplication, and let \[V= \e_{n+1} G \setminus \{0\}.\] To better understand the set $V$, note that the representation of $\SO(n,1)$ we are using is \[\SO(n,1) = \{A \in \SL_{n+1}(\R) : AJA^T = J\},\] where \[J=\begin{pmatrix} 0 & 0 & 1 \\0 & -I_{n-1} & 0 \\ 1 & 0 & 0 \end{pmatrix}.\] 

Let $P$ be such that  \[J':=\begin{pmatrix} -I_{n} & 0 \\ 0 & 1  \end{pmatrix}=PJP^T.\] Then $VP$ is the upper half of the ``light cone'' in the standard representation of $\SO(n,1)$. In particular, this consists of null vectors of $$Q'(x_1,\ldots,x_{n+1}) = x_{n+1}^2 - x_1^2 - \cdots -x_{n}^2$$ with $x_{n+1}>0$. In our case, $V$ consists of null vectors of $$Q(x_1,\ldots,x_{n+1}) = 2x_1x_{n+1}-x_2^2-\cdots-x_{n}^2.$$ %with $x_1,x_{n+1}>0.$

%We have $G=KAU$ \[\begin{pmatrix}1 & 0 \\0 & SO(n)\end{pmatrix}\cdot a \cdot e_{n+1}=(c, c\cdot\text{the last row of }SO(n).\]$KK^T=I$\[x_1^2-x_2^2-\dots-x_{n+1}^2=0.\]

%\[\SO(n,1)=\{A\::\:AJA^T=J\}\]
%If $xJx^T=0$ then for $A\in\SO(
%n,1)$, we have $xAJ(xA)^T=0$. \[
%J=\begin{pmatrix} 0 & 0 & 1 \\0 & -I_{n-1} & 0 \\ 1 & 0 & 0 \end{pmatrix}. \]
%There should be $P$ such that \[
%J':=\begin{pmatrix} -I_{n} & 0 \\ 0 & 1  \end{pmatrix}=PJP^T.\]
%This $P$ is \[
%P=\begin{pmatrix} 1 & 0 & \cdots & 0 & 1\\
%0 & 1 & \cdots & 0 & 0\\
%0 & 0 & \ddots & 1 & 0\\
%-1 & 0 & \cdots & 0 & 1
%\end{pmatrix}. \]
%Then $(xP)J'(xP)^T=0$

%Let $G$ act on $\R^{n+1}$ by right matrix multiplication. Define \[V = \e_{n+1} G \setminus \{0\}.\] Note that $V$ is contained in the null set of the quadratic form \[Q(x_1,\dots,x_{n+1}) = 2x_1x_{n+1} - x_2^2 - \cdots - x_{n+1}^2.\]

%Note that we are using a non standard representation of $\SO(n,1)$, and the set $V$ in the usual representation would be the upper half of the light cone.

%Let $G$ act on $\R^{n+1}$ by right matrix multiplication. Observe that $G$ acts transitively on $V$, the null set of the quadratic equation \[Q(x_1,\dots,x_{n+1})=2x_1 x_{n+1}-x_2^2\dots-x_{n}^2\] in $\R^{n+1}\setminus \{0\}$. Using the previous section, we will obtain Proposition \ref{prop: vector asymptotic intro} in this setting, which is restated below:

\begin{proposition}
    Let $\Gamma$ be convex cocompact. For any $\ov\varphi \in C_c(V)$ and every $v \in V$ with $v^- \in \Lambda_r(\Gamma),$ as $T\to \infty$, we have that \[\frac{1}{T^{\delta_\Gamma/2}} \sum\limits_{\gamma \in \Gamma_T} \ov\varphi(v\gamma) \asymp \int_{V}\ov\varphi(u)\frac{d\ov\nu(u)}{(\norm{v}_2\norm{u}_2)^{\delta_\Gamma/2}},\] where the implied constant depends on $v$ and $\Gamma$.
\end{proposition} The measure $\ov\nu$ is described more explicitly in \eqref{eq: nice ov nu defn}, below.

\bigskip

Let \[\e_{n+1} = (0,\ldots,0,1)\in \R^{n+1}.\] Then \begin{equation}
    \Stab_G(\e_{n+1}) = UM,
\end{equation} and hence \begin{equation} A \times M \backslash K\cong UM \backslash G \cong V \end{equation} via right matrix multiplication\[UMg \mapsto \e_{n+1}g.\]

We will now interpret Corollary \ref{cor: main} in this setting. 
We start by understanding the measure $\ov\nu$. 

We view $V$ as $(M\backslash K)\times \R^+,$ via the ``polar decomposition'' of $v \in V$,  \begin{equation}\label{eq: v identified with ak}
    v = \|v\|_2 \e_{n+1} k=\e_{n+1}a_{-\log\|v\|_2}k,
\end{equation} 
where $\R^+ = \{r \in \R : r > 0\}$ and $\|\cdot\|_2$ denotes the Euclidean norm on $V$.  
We may also identify $M\backslash K$ with $\partial(\H^{n})$ via% $\mathbb{S}^n$ via
\begin{equation}\label{eq: identifying K mod M with Sn}   
    Mk \mapsto w_o^- k.
    %Mk \mapsto e_{n+1} k,
\end{equation} 

%where we use the unit disc model for $\H^{n+1}$. %, so $w_o^- \in \mathbb{S}^n \cong \partial(\H^{n+1})$. 
%$A$ is identified with $\R^+$ via \begin{equation}\label{eq: identifying A with Rplus}   a_t \mapsto e^{-t}.\end{equation}

%Let $\ov\varphi \in C_c^\infty(\R^{n+1}\setminus\{0\})$. %We wish to define $\varphi\in C_c^\infty(UM \backslash G)$ and a measure $\ov\nu$ on $\R^{n+1}\setminus\{0\}$ such that \[\nu(\varphi\circ\theta\circ\tau) = \ov\nu(\ov\varphi).\]

%We can define $\varphi\in C_c^\infty(UM\backslash G)$ by \[\varphi(h) = \ov\varphi(\e_{n+1} h)\]
%for any $h \in UM \backslash G.$

Thus, given $v \in V$, \eqref{eq: v identified with ak} and \eqref{eq: identifying K mod M with Sn} uniquely determine a pair $(a_{-\log\|v\|_2},Mk)\in A\times M\backslash K$, or equivalently, a pair $(a_{-\log\|v\|_2}, w_o^-k)\in A\times\partial(\H^n)$. 
%Then, we can assume that the representative $x\in U\backslash G$, $\e_{n+1}x=v$, satisfies $\Psi(x)=a_{\log\|v\|}k$. 

Viewing $\partial(\H^n)$ as $M\backslash K$ as in \eqref{eq: identifying K mod M with Sn}, we may in turn identify this with $\mathbb{S}^n \subseteq \R^{n+1}$ via \[ 
    w_o^-k \mapsto \e_{n+1}k.
\] Thus, $\nu_o$ uniquely determines a measure $\ov{\nu}_o$ on $\mathbb{S}^n \cap V$ via \begin{equation}
    d\ov{\nu}_o(\e_{n+1}k) = d\nu_o(w_o^-k). \label{eq: ov nu o defn}
\end{equation}

Then, since $K$ stabilizes $o$ and $M$ stabilizes $w_o$, $\ov\nu$ can be described from \eqref{eq; defn of nu}: if $s = \beta_{(a_{-\log\|v\|_2}k)^-}(o, a_{-\log\|v\|_2}k(o))=\log\norm{v}_2$, 
\begin{align*}
	d\ov\nu(v)&:=d\nu(a_{-\log\|v\|_2}k)\\
	&= e^{\delta_\Gamma \beta_{(a_{-\log\|v\|_2}k)^-}(o, a_{-\log\|v\|_2}k(o))} d\nu_o(w_o^-a_{-\log\|v\|_2}k)ds\\ &=e^{\delta_\Gamma s} d\nu_o(w_o^-k)ds\\
	&=\norm{v}_2^{\delta_\Gamma-1} d\ov{\nu}_o(\e_{n+1}k)d\|v\|_2.
\end{align*} For $v \in V$, define \[v^-:= \e_{n+1}k \in \mathbb{S}^n,\] where $v$ corresponds to $(a_{-\log\|v\|_2},Mk) \in A \times M\backslash K.$ Then we have \begin{equation}
    d\ov{\nu}(v)= \|v\|_2^{\delta_\Gamma-1} d\ov{\nu}_o(v^-)d\|v\|_2. \label{eq: nice ov nu defn}
\end{equation}

As discussed in the previous section, $v\star u$ may be computed by the formula in \eqref{eq: star} for any choice of representatives of $v$ and $u$ in $U\backslash G$. In particular, if \[v=\|v\|_2\e_{n+1}k_v,\quad u=\|u\|_2\e_{n+1}k_u,\] then \[v\star u = \sqrt{\frac{1}{2}\|v\|_2\|u\|_2 \max\limits_{1\le i,j\le n+1} \left|(k_v\inv)_{i,1}(k_u)_{n+1,j}\right|},\]
where $k_{i,j}$ denotes the $(i,j)$ entry of $k$.
In particular\[
v\star u\asymp \sqrt{\norm{v}_2\norm{u}_2}.\]

Putting this together with Corollary \ref{cor: asymptotic in general application} yields the proposition.

\subsection{Wedge products}

The previous example can be generalized to $\bigwedge^j \R^{n+1}$ for any $1\le j\le n$. Fix $j$, and let \[
W=\bigwedge^j \R^{n+1},\quad\text{and}\quad v_0=v_0(j)=e_{n-j+1}\wedge\cdots\wedge e_{n+1}, \]
with $G$ acting on $W$ by right multiplication. Then, 
\[\operatorname{Stab}_{e_{n-j+1}\wedge\cdots\wedge e_{n+1}}=U\cdot M_j\]
for some $M_j\subseteq M$. Define \[V = v_0 G \setminus \{0\}.\]
Fix a norm on $V$ which is invariant under $K$ such that $\norm{v_0}=1$.   

Since any $v\in V$ can be written as \[
v=v_0a_{-\log\norm{v}}k,\]
where $k\in M_j\backslash K$, in a similar way to the construction in the previous section, one can show that if $a_{-\log\|v\|}k \in UP$ and can be written as $uamv \in UAM\tilde U$, then \[
d\ov\nu(v)=\norm{v}^{\delta_\Gamma-1}d\nu_o(v^-)d\norm{v}dm,\]
where $v^-:=w_o^- k$, and $dm$ is the push forward of the probability Haar measure on $M_j\backslash M$. $d\ov\nu(v)$ is zero if $a_{-\log\|v\|}k\not\in UP$, because the original measure $\nu$ is supported on $P$.

Moreover, by reasoning in the beginning of \S\ref{section: applications}, $v\star u$ is well defined and, as in the previous section, we have that \[
v\star u\asymp \sqrt{\norm{v}\norm{u}}.\]

\end{document}